\documentclass[11pt,reqno]{amsart}

\usepackage{indentfirst, amssymb, amsmath, amsthm, mathrsfs, setspace, enumerate}

\usepackage[
bookmarksnumbered,
colorlinks,
plainpages,
pdfencoding=auto,
psdextra
]{hyperref}

\usepackage{xcolor}
\usepackage{graphicx}
\usepackage{float}
\usepackage{caption}
\usepackage{subcaption}
\usepackage{cite}

\usepackage{tikz}
\usepackage{pgfplots}
\pgfplotsset{compat=1.18}

\usetikzlibrary{
	shapes,
	arrows,
	positioning,
	fit,
	backgrounds,
	calc,
	shapes.geometric,
	shadows,
	arrows.meta
}

\definecolor{color1}{RGB}{70,40,200}
\definecolor{color2}{RGB}{213,14,0}
\definecolor{color3}{RGB}{0,158,10}
\definecolor{color4}{RGB}{230,126,34}

\newtheorem*{theoA}{Theorem A}
\newtheorem*{theoB}{Theorem B}

\newtheorem{theo}{Theorem}[section]
\newtheorem{lem}{Lemma}[section]
\newtheorem{cor}{Corollary}[section]

\newtheorem{defi}{Definition}[section]
\newtheorem{rem}{Remark}[section]

\newcommand{\ol}{\overline}
\newcommand{\be}{\begin{equation}}
	\newcommand{\ee}{\end{equation}}
\newcommand{\beas}{\begin{eqnarray*}}
	\newcommand{\eeas}{\end{eqnarray*}}
\newcommand{\bea}{\begin{eqnarray}}
	\newcommand{\eea}{\end{eqnarray}}

\numberwithin{equation}{section}

\begin{document}
\title[E\MakeLowercase {stimates for a Parameterized}......]{\LARGE E\Large\MakeLowercase {stimates for a Parameterized} H\MakeLowercase {ankel Functional of Starlike Mappings of Order $\alpha$ and Quasi-Convex Mappings of Type $B$ in Complex} B\MakeLowercase{anach Spaces} }

\date{}
\author[S. M\MakeLowercase{ajumder}, A. B\MakeLowercase{anerjee}, S. P\MakeLowercase{anja} \MakeLowercase{and} D. P\MakeLowercase{ramanik} ]{S\MakeLowercase{ujoy} M\MakeLowercase{ajumder}$^{1}$, A\MakeLowercase{bhijit} B\MakeLowercase{anerjee}$^2$, S\MakeLowercase{hantanu} P\MakeLowercase{anja}$^3$$^*$ \MakeLowercase{and} D\MakeLowercase{ebabrata} P\MakeLowercase{ramanik}$^4$}

\address{$^{1}$ Sujoy Majumder, Department of Mathematics, Raiganj University, Raiganj, West Bengal-733134, India.}
\email{sm05math@gmail.com, sjm@raiganjuniversity.ac.in}
\address{$^{2}$ Abhijit Banerjee, Department of Mathematics, University of Kalyani, West Bengal 741235, India.}
\email{abanerjee\_kal@yahoo.co.in, abhijitbanerjee@klyuniv.ac.in}
\address{$^3$ Shantanu Panja, Department of Mathematics, University of Kalyani, West Bengal 741235, India.}
\email{panjasantu07@gmail.com, shantanumath26@klyuniv.ac.in}
\address{$^{4}$ Debabrata Pramanik, Department of Mathematics, Raiganj University, Raiganj, West Bengal-733134, India.}
\email{debumath07@gmail.com}

\renewcommand{\thefootnote}{}
\footnote{2020 \emph{Mathematics Subject Classification}: Primary 32H02· Secondary 30C45.}
\footnote{\emph{Key words and phrases}: Starlike mappings of order $\alpha$, Quasi-convex mappings of type $B$, Fr\'echet derivatives, Coefficient problems, Hankel determinant.}
\footnote{*\emph{Corresponding Author}: Shantanu Panja.}

\renewcommand{\thefootnote}{\arabic{footnote}}
\setcounter{footnote}{0}

\begin{abstract}
In this paper, we introduce and investigate a parameterized Hankel-type functional
	$$H_{2,2}^{\mu}(f)=\mu a_2a_4-a_3^2,\qquad \mu\in[0,1],
	$$	which reduces to the classical second-order Hankel determinant $H_{2,2}(f)$ when $\mu=1$. We obtain estimates for $H_{2,2}^{\mu}(f)$ over the class of starlike mappings of order $\alpha$ on the unit ball of a complex Banach space and on the unit polydisk $\mathbb{U}^n$ in $\mathbb{C}^n$. We also establish estimates for $H_{2,2}^{\mu}(f)$ over the class of quasi-convex mappings of type $B$ on the unit ball of a complex Banach space and on the unit polydisk $\mathbb{U}^n$. For $\mu=1$, our results recover the corresponding estimates for the classical Hankel determinant and extend the known one-dimensional results of Cho et al. \cite{Cho-2018} and Thomas et al. \cite{Thomas-Tuneski} to higher-dimensional settings.

\end{abstract}
\thanks{Typeset by \AmS -\LaTeX}
\maketitle
\section{\bf Introduction}

Univalent function theory has been a core part of complex analysis since  its inception, with the Riemann mapping theorem as its foundational result. However, Poincaré demonstrated that this elegant result fails in higher dimensions.  As a result, the classical methods from the one-dimensional setting require substantial modification. This challenge has inspired extensive research over the years to identify appropriate subclasses and structural assumptions that enable meaningful extensions to higher-dimensional spaces.

Let $\mathcal{S}(\mathbb{U})$ denote the class of functions $f$ that are univalent in the open unit disk $\mathbb{U} = \{z \in \mathbb{C} : \vert{}z\vert{} < 1\}$, normalized by $f(0) = 0$ and $f'(0) = 1$. Any function $f \in \mathcal{S}(\mathbb{U})$ has a power series expansion of the form:

\begin{align}\label{eq1}
	f(z) = z + \sum_{n=2}^{\infty} a_n z^n, \quad z \in \mathbb{U}.
\end{align}

The Hankel determinant $H_{2,2}(f)$ of a function $f \in \mathcal{S}(\mathbb{U})$ of the form
\eqref{eq1} is defined as
\begin{align*}
H_{2,2}(f) = \begin{vmatrix} a_2 & a_3 \\ a_3 & a_4 \end{vmatrix} =a_2a_4 - a_3^2.
\end{align*}
The Hankel determinant serves as a powerful invariant that encodes subtle geometric information about the function's behavior. Unlike individual coefficients, which may vary substantially, the combination $a_2a_4 - a_3^2$ captures the nonlinear interplay between consecutive coefficients and has proven remarkably effective in distinguishing between different geometric subclasses (see, e.g., \cite{Bansal-2015, Cho-2017, Janteng-2006, Lee-2013}).
Given $\alpha \in [0, 1)$, by $\mathcal{S}^*_{\alpha}(\mathbb{U})$ we denote a subclass of $\mathcal{S}(\mathbb{U})$ of functions $f$ such that
\begin{align*}
\Re\left\lbrace\frac{z f'(z)}{f(z)} \right\rbrace > \alpha,\quad z \in \mathbb{U}
\end{align*}
and is called the class of starlike mappings of order $\alpha$.

The study of starlike mappings represents one of the most extensively developed areas in geometric function theory, with applications ranging from fluid dynamics to conformal mapping problems. The recent sharp estimate of the Hankel determinant $H_{2,2}(f)$ over the class $\mathcal{S}^*_{\alpha}(\mathbb{U})$ by Cho et al. \cite{Cho-2018} represents a significant milestone:
\begin{theoA}\emph{\cite[Theorem 2.3]{Cho-2018}}
Let $\alpha \in [0, 1)$. If $f \in \mathcal{S}^*_{\alpha}(\mathbb{U})$ is of the form \eqref{eq1}, then
\begin{align}\label{TA1}
|H_{2,2}(f)| \le (1 - \alpha)^2.
\end{align}
The inequality \eqref{TA1} is sharp, and the equality holds for the function
\begin{align*}
f(z) := \frac{z}{(1 - z^2)^{1-\alpha}}, \quad z \in \mathbb{U}.
\end{align*}
\end{theoA}
Let $\mathcal{K}(\mathbb{U})$ denote the familiar subclass of $\mathcal{S}$ consisting of functions which are convex in $\mathbb{U}$. Then $f\in\mathcal{K}(\mathbb{U})$ if and only if $\Re\left( 1+\frac{z f''(z)}{f'(z)} \right) > 0$, $z \in \mathbb{U}$.
\par For the estimate of the Hankel determinant $H_{2,2}(f)$ over the class $\mathcal{K}(\mathbb{U})$, we have the following result.
\begin{theoB}\emph{\cite[Theorem 4.2.8]{Thomas-Tuneski}} Let $f\in\mathcal{K}(\mathbb{U})$. Then
\begin{align*}
	|H_{2,2}(f)|\leq \frac{1}{8}.
\end{align*}	
\end{theoB}
\begin{figure}[htbp]
\centering
\begin{tikzpicture}[scale=0.85]
	\begin{axis}[
		width=12cm, height=7cm,
		xlabel={Parameter $\alpha$},
		ylabel={Sharp bound for $|H_{2,2}(f)|$},
		title={Known sharp Hankel-determinant bounds (Theorems A and B)},
		title style={font=\small},
		legend style={
			at={(.95,0.90)},
			anchor=north east,
			font=\small
		},
		grid=both,
		grid style={line width=.1pt, draw=gray!30},
		major grid style={line width=.2pt,draw=gray!50},
		xmin=0, xmax=1, ymin=0, ymax=1.1,
		]
		\addplot[domain=0:1, samples=100, thick, color=blue] { (1-x)^2 };
		\addlegendentry{Starlike order $\alpha$, Thm.\ A}
		\addplot[domain=0:1, samples=100, dashed, thick, color=red] {0.125};
		\addlegendentry{Convex, Thm.\ B ($\alpha$-independent)}
	\end{axis}
\end{tikzpicture}
\caption{The paper extends two classical sharp bounds to higher dimensions: Theorem A gives the starlike bound $(1-\alpha)^2$, while Theorem B gives the constant $1/8$ for the convex class.
	Sections~\ref{Sec-3} onward show that the starlike bound remains unchanged on arbitrary complex Banach spaces and the unit polydisk $\mathbb U^n$; see Figure~\ref{fig:dimension_invariance}.}
\label{fig:bound_comparison_1}
\end{figure}
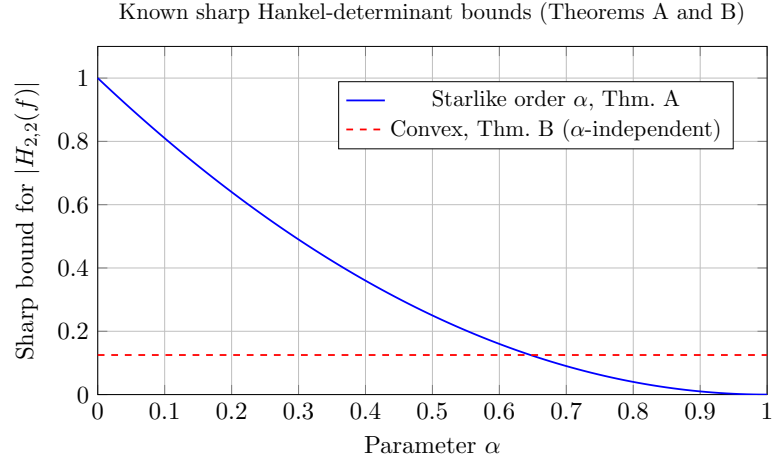

\vspace{1.2mm}

Before proceeding further, let us establish some fundamental terminology that will recur throughout this work. Let $X$ denote a complex Banach space with norm $||\cdot||$. The ball of radius $r$ is denoted by $\mathbb{B}_r =\{z\in X:||z||<r\}$. When $r=1$, we will simply use $\mathbb{B}$ and if $X=\mathbb{C}$, then $\mathbb{B}=\mathbb{U}$. Let $\partial \mathbb{B}$ denote the boundary of $\mathbb{B}$, $\overline{\mathbb{B}}$ its closure, and $\partial_0 \mathbb{U}^n$ the characteristic boundary (i.e., the boundary on which the maximum modulus of a holomorphic function is attained) of the unit polydisk $\mathbb{U}^n$.

Now let $X$ and $Y$ be two complex Banach spaces with respect to the norms $||.||_1$ and $||.||_2$ respectively. For simplicity, we denote both norms by $||.||$, when there is no possibility of confusion. Let $L(X, Y)$ denote the set of continuous linear operators from a complex Banach space $X$ into a complex Banach space $Y$ with the standard operator norm 
\begin{align*} 
	||A||=\sup\{||A(z)||: ||z||=1,\quad A\in L(X,Y)\}.
\end{align*} 

Let $I$ be the identity operator in $L(X, X)$. For each $z \in X \setminus \{0\}$, define the set  
\begin{align*}
	\displaystyle 
	T(z) = \{ l_z \in L(X, \mathbb{C}) : l_z(z) = \Vert{}z\Vert{}, \quad \Vert{}l_z\Vert{} = 1 \}.
\end{align*}

This set is guaranteed to be nonempty by the Hahn-Banach theorem. 
Let $\Omega$ be a domain in $X$ and let $f:\Omega \rightarrow Y$ be a mapping. Then $f$ is called holomorphic on $\Omega$ if for any $z\in\Omega$ there is a mapping $Df(z)\in L(X,Y)$, called the Fr\'echet derivative of $f$ at $z$ such that 
\begin{align*} \lim_{h\to 0} \frac{||f(z+h)-f(z)-Df(z)h||}{||h||}=0,\end{align*} i.e., 
\begin{align*} 
	f(z+h)=f(z)+Df(z)h+o(||h||).
\end{align*}

Let $\mathcal{H}(\Omega, \Omega')$ represent the set of all holomorphic mappings from a domain $\Omega \subset X$ into another domain $\Omega' \subset Y$. We simply write $\mathcal{H}(\Omega, X) = \mathcal{H}(\Omega)$. It is well known that if $f \in \mathcal{H}(\Omega,Y)$, then for all $w$ in some neighborhood of $z \in \Omega$, the function can be expressed as the power series expansion:
\begin{align*}
	\displaystyle
	f(w) = \sum_{n=0}^{\infty} \frac{1}{n!} D^n f(z)((w - z)^n),
\end{align*}
where $D^n f(z)$ is a bounded symmetric $n$-linear mapping from $\prod_{j=1}^n X$ into $Y$, called the $n$-th Fr\'echet derivative of $f$ at $z$. Moreover, the term $D^n f(z)((w - z)^n)$ is shorthand for applying the $n$-linear mapping to $n$ copies of the vector $(w - z)$:
\begin{align*}
	\displaystyle
	D^n f(z)((w - z)^n) = D^n f(z)(\underbrace{w - z, w - z, \dots, w - z}_{n \text{ times}})
\end{align*}
where $D^0f(z)=f(z)$.\vspace{1.2mm}

A mapping $f \in \mathcal{H}(\mathbb{B})$ is said to be normalized if $f(0) = 0$ and $Df(0) = I$, where $I \in L(X,X)$ is the identity operator. It is called biholomorphic if $f$ is bijective with a holomorphic inverse $f^{-1}$ on $f(\mathbb{B})$ and locally biholomorphic if $Df(z)$ has a bounded inverse for every $z \in \mathbb{B}$. The family of all normalized biholomorphic mappings on $\mathbb{B}$ is denoted by $\mathcal{S}(\mathbb{B})$. \vspace{1.2mm}

A case of particular interest is that of the $p$-norm in
$\mathbb{C}^{n}$, that is,
\begin{align*}
	\displaystyle
	\|z\|_{p}=
	\begin{cases}
		\left(
		\displaystyle\sum_{j=1}^{n}|z_{j}|^{p}
		\right)^{1/p},
		& 1\le p<\infty, \\[2ex]
		\displaystyle\max_{1\le j\le n}|z_{j}|,
		& p=\infty.
	\end{cases}
\end{align*}

Let $\mathbb{B}(p)$ denote the unit ball with respect to $\|\cdot\|_{p}$.
When $p=\infty$, $\mathbb{B}(p)$ reduces to the unit polydisk $\mathbb{U}^n$ of $\mathbb{C}^{n}$. It is known that if $X=\mathbb{C}^{n}$ with respect to the maximum norm $\Vert{}.\Vert{}_\infty $, and $l_z\in T(z)$, then
\begin{align}\label{eq3}
\displaystyle
l_z(w)=
\sum_{\{j:\|z\|_{\infty}=|z_j|\}}
t_j\,
\frac{w_j\overline{z_j}}
{\|z\|_{\infty}},
\end{align}
where each $t_j\geq0$ and $\sum\limits_{\{j:\|z\|_{\infty}=|z_j|\}} t_j=1.$\vspace{1.2mm}

The following family plays a key role in our discussion:
\begin{align*}
	\displaystyle
	\mathcal{M}=\{f\in\mathcal{S}(\mathbb{B}),\; \Re\{l_z(f(z))\}>0,\; z\in\mathbb{B}\backslash\{0\},\; l_z\in T(z)\}.
\end{align*}
In the case of the maximum norm $\Vert{}.\Vert{}_\infty $ on $\mathbb{C}^n$, the class $\mathcal{M}$ consists of all normalized holomorphic mappings $w:\mathbb{U}^n \to \mathbb{C}^n$ satisfying
\begin{align*}
	\displaystyle
	\Re\left\lbrace \frac{w_j(z)}{z_j} \right\rbrace >0
\end{align*}
whenever $\Vert{}z\Vert{}_\infty = \vert{}z_j\vert{} > 0$ for $1 \le j \le n$.\vspace{1.4mm}

In several complex variables, classical results from geometric function theory often fail to hold in their original form. Most notably, Poincar\'e disproved the higher-dimensional analogue of the Riemann mapping theorem, while Cartan constructed counterexamples to the Bieberbach conjecture on the unit ball. These fundamental obstacles have necessitated the development of new methodologies and the identification of appropriate subclasses that preserve the essential features of the classical theory.

As a result, extending univalent function theory to higher dimensions requires introducing appropriate subclasses or additional structural restrictions. Recent years have witnessed remarkable progress in establishing coefficient bounds for various subclasses of normalized biholomorphic mappings in higher dimensions. Notable contributions in this direction include the works of Bracci et al. \cite{Bracci-2016}, Bracci and Roth \cite{Bracci-2018}, Graham et al. \cite{Graham-2002}, \cite{Graham-2016}, Hamada and Honda \cite{Hamada-Honda-2008}; Hamada and Kohr \cite{Hamada-Kohr-2020}.
\begin{figure}[htbp]
	\centering
	\begin{tikzpicture}[scale=0.9]
		\begin{axis}[
			width=12cm, height=6.5cm,
			xlabel={Dimension $n$ (or ``$n=\infty$'' for a general Banach space $X$)},
			ylabel={Sharp constant in the bound},
			title={The sharp constant $(1-\alpha)^2$ is independent of dimension},
			title style={font=\small},
			xtick={1,2,3,4,5,6},
			xticklabels={$1$,$2$,$3$,$4$,$5$,$X$ (Banach)},
			ymin=0, ymax=1.1, xmin=0.5, xmax=6.5,
			grid=both,
			grid style={line width=.1pt, draw=gray!30},
			legend pos=south east,
			]
		
			\addplot[thick, color=blue, mark=*, mark size=2.2pt, only marks]
			coordinates {(1,1)(2,1)(3,1)(4,1)(5,1)(6,1)};
			\addlegendentry{Sharp constant $(1-\alpha)^2/(1-\alpha)^2=1$ (normalized)}
			\addplot[dashed, gray, thick] coordinates {(0.5,1)(6.5,1)};
			\node[draw, fill=white, rounded corners, font=\footnotesize, align=left] at (axis cs:3.3,0.55)
			{Thm.\ A ($n=1$), Cor.\ \ref{Cr-22} ($n\ge2$),\\
				Cor.\ \ref{Cr-21} (general Banach space $X$):\\
				all give the \emph{same} sharp constant $(1-\alpha)^2$};
		\end{axis}
	\end{tikzpicture}
	\caption{Unlike many higher-dimensional extensions, the sharp bound here for the generalized Hankel determinant is $(1-\alpha)^2$, up to the natural factor $|z|_\infty^6$.
		Remarkably, this constant is independent of dimension, holding on $\mathbb U$, $\mathbb U^n$ for every $n\ge2$, and the unit ball of any complex Banach space $X$.
		Values are normalized to $1$ to emphasize this dimension-independence.}
	\label{fig:dimension_invariance}
\end{figure}
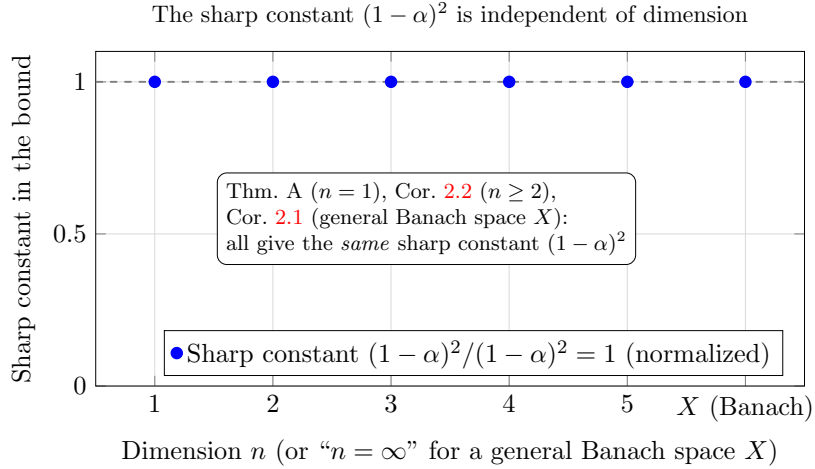

$\bullet$ Inspired by the Hankel determinant $H_{2,2}(f)$, we introduce the following generalized Hankel functional depending on a parameter $\mu\in[0,1]$:  $$H^{\mu}_{2,2}(f)=\mu a_2a_4- a_3^2,\qquad \mu\in[0,1].$$
$H^{\mu}_{2,2}(f)$ reduces to $H_{2,2}(f)$ when $\mu=1$.  

Clearly, $H_{2,2}^{\mu}(f)$ reduces to the classical Hankel determinant $H_{2,2}(f)$ when $\mu=1$. Motivated by this generalization, we study the sharp estimates of $H_{2,2}^{\mu}(f)$ for several important classes of holomorphic mappings in higher-dimensional settings. In particular, we extend the notion of starlike mappings and quasi-convexity mappings of type $B$ in complex Banach space, generalizing the classical definition recalled in Theorem A and Theorem B above. More precisely, we obtain the following results:

\par\emph{$(i)$} We obtain the sharp estimate of the Hankel functional $H_{2,2}^{\mu}(f)$ for starlike mappings of order $\alpha$ on the unit ball in a complex Banach space and on the unit polydisk $\mathbb{U}^n$ in $\mathbb{C}^n$.

\par\emph{$(ii)$} We establish the sharp estimate of the Hankel functional $H_{2,2}^{\mu}(f)$ for quasi-convex mappings of type $B$ on the unit ball in a complex Banach space as well as on the unit polydisk $\mathbb{U}^n$ in $\mathbb{C}^n$.
\section{\bf Bounds for $H_{2,2}^{\mu}(f)$ over the class of starlike mappings of order $\alpha$ in complex Banach spaces}\label{Sec-3} 
A biholomorphic mapping on $\mathbb{B}$ is said to be starlike if $f(\mathbb{B})$ is starlike with
respect to the origin. If $\mathbb{B}$ is the unit ball of $X$, a complex Banach space, we let $\mathcal{S}^*(\mathbb{B})$ denote the subset of $\mathcal{S}(\mathbb{B})$ consisting of normalized starlike mappings from $\mathbb{B}$ into $X$.

In this section, we recall the definition of starlikeness of order $\alpha$ in complex Banach space.
\begin{defi}\emph{\cite[Definition 1.2]{Hamada-2001}}
	Let $f:\mathbb{B}\to X$ be a normalized locally biholomorphic mapping and let $\alpha\in [0,1)$. We say that $f$ is starlike of order $\alpha$ if 
	\begin{align}\label{SL1}
		\Re\left\lbrace \frac{||z||}{l_z([Df(z)]^{-1}f(z))}\right\rbrace>\alpha,\quad z\in\mathbb{B}\backslash\{0\},\quad l_z\in T(z).
	\end{align}
\end{defi}

Let $\mathcal{S}^*_{\alpha}(\mathbb{B})$ denote the class of starlike mappings of order $\alpha$ on $\mathbb{B}$.\vspace{1.2mm}

 There are now many references for starlike mappings in higher dimensions. For more information, see the recent monographs by Gong  \cite{Gong} and Kohr \& Liczberski  \cite{Kohr-Liczberski-1998}. Some foundational and original papers in this area include 
 \cite{Chu}, \cite{SSQ_CHMS_1990}, \cite{Su2}-\cite{Xu-He-Xu}.
 \vspace{1.1mm}
 
The following results provide bounds of the parameterized Hankel functional for the class of starlike mappings of order $\alpha$ in higher dimensions.
 
 \begin{theo}\label{Th-21}Let $f\in\mathcal{H}(\mathbb{B},\mathbb{C})$ with $f(0)=1$ and $F(z)=zf(z)\in\mathcal{S}^*_{\alpha}(\mathbb{B})$ and let $\mu\in [0,1]$. Then
 \begin{align*}
 	\left|\mu\frac{l_{z}(D^2F(0)(z^2))}{2!\|z\|^2}\frac{l_{z}(D^4F(0)(z^4))}{4!\|z\|^4}-\left(\frac{l_z(D^3F(0)(z^3))}{3!\|z\|^3}\right)^2\right|\le{(1-\alpha)^2}(9-8\mu).
 \end{align*}
 \end{theo}
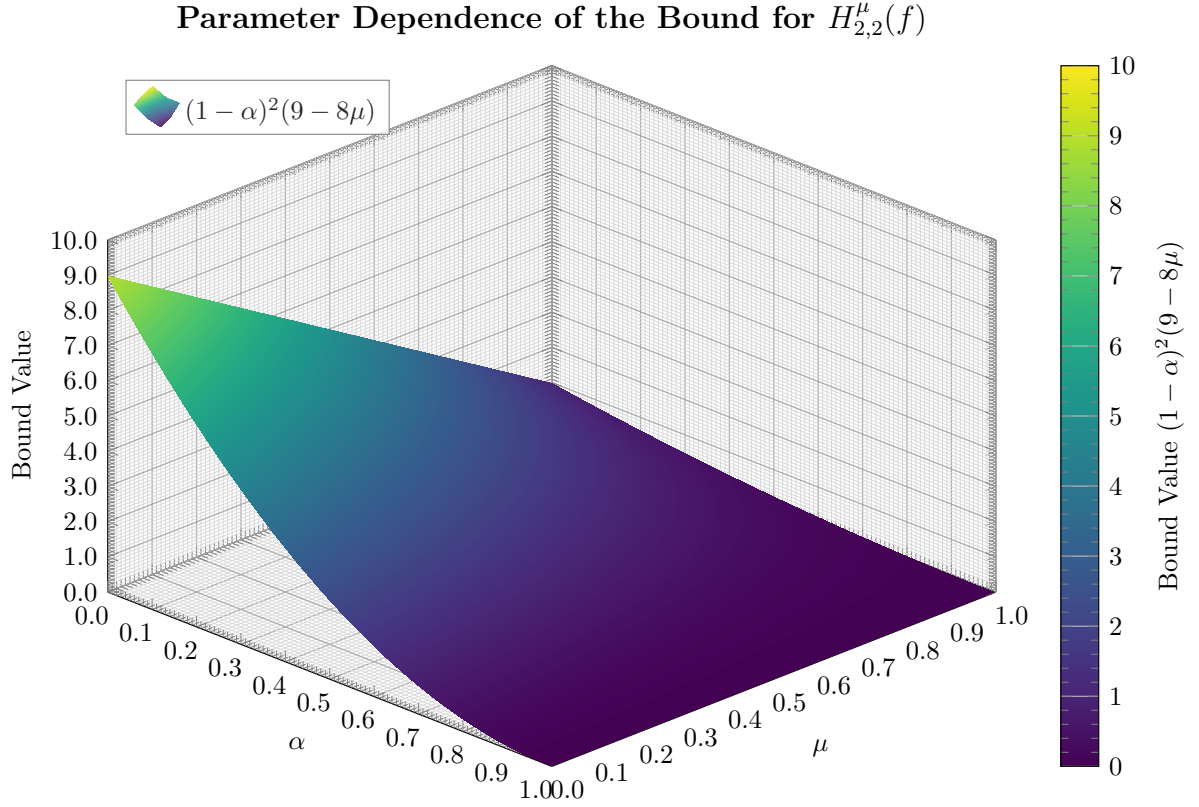
\begin{figure}[htbp]
	\centering
	\begin{tikzpicture}
		\begin{axis}[
			title={\textbf{Parameter Dependence of the Bound for $H_{2,2}^{\mu}(f)$}},
			xlabel={$\alpha$},
			ylabel={$\mu$},
			zlabel={\;\;\;Bound Value},
			xmin=0, xmax=1,
			ymin=0, ymax=1,
			zmin=0, zmax=10,
			colormap/viridis,
			view={45}{35},
			width=0.86\textwidth,
			height=0.7\textwidth,
			grid=major,
			grid style={line width=0.5pt, draw=gray!65},
			minor grid style={line width=0.15pt, draw=gray!30},
			xmajorgrids=true,
			ymajorgrids=true,
			zmajorgrids=true,
			xminorgrids=true,
			yminorgrids=true,
			zminorgrids=true,
			minor tick num=9,
			ticklabel style={font=\small, /pgf/number format/.cd, fixed, fixed zerofill, precision=1},
			xlabel style={font=\small, anchor=north},
			ylabel style={font=\small},
			zlabel style={font=\small},
			title style={font=\large\bfseries},
			colorbar,
			colorbar style={
				ylabel={Bound Value $(1-\alpha)^2(9-8\mu)$}, 
				ylabel style={font=\small},
				ticklabel style={font=\small},
				grid=major,
				grid style={line width=0.2pt, draw=gray!40},
				minor tick num=4
			},
			legend style={
				at={(0.02,0.98)}, 
				anchor=north west, 
				font=\small, 
				fill=white, 
				fill opacity=0.85,
				draw=black!50
			},
			samples=50,
			shader=interp,
			point meta min=0,
			point meta max=10,
			xtick={0,0.1,0.2,0.3,0.4,0.5,0.6,0.7,0.8,0.9,1},
			ytick={0,0.1,0.2,0.3,0.4,0.5,0.6,0.7,0.8,0.9,1},
			ztick={0,1,2,3,4,5,6,7,8,9,10}
			]
			\addplot3[
			surf,
			shader=interp,
			domain=0:1,
			y domain=0:1
			] { min(9, (1-x)^2 * (9 - 8*y)) };
			\addlegendentry{$(1-\alpha)^2(9-8\mu)$}

		\end{axis}
	\end{tikzpicture}
	\caption[Parameter dependence of the sharp bound]{%
		\textbf{Parameter dependence of the upper bound} 
		$(1-\alpha)^2(9-8\mu)$ in Theorem~\ref{Th-21} for 
		$0\leq\alpha,\mu\leq 1$. The dense grid structure enhances the 
		three-dimensional visualization of the \textbf{quadratic decay} 
		with respect to $\alpha$ and \textbf{linear decay} with respect to $\mu$.%
	}
	\label{fig:bound_surface}
\end{figure}

 From Theorem \ref{Th-21}, we get the following corollary.
  
 \begin{cor}\label{Cr-21}Let $f\in\mathcal{H}(\mathbb{B},\mathbb{C})$ with $f(0)=1$ and $F(z)=zf(z)\in\mathcal{S}^*_{\alpha}(\mathbb{B})$. Then
 	\begin{align*}
 		\left|\frac{l_{z}(D^2F(0)(z^2))}{2!\|z\|^2}\frac{l_{z}(D^4F(0)(z^4))}{4!\|z\|^4}-\left(\frac{l_z(D^3F(0)(z^3))}{3!\|z\|^3}\right)^2\right|\le (1-\alpha)^2.
 	\end{align*}
 	The bound is sharp.
 \end{cor}

 To establish the sharpness of the bound of Corollary \ref{Cr-21}, let us consider the function $F$ defined by
 \begin{align*}
 	F(z)=zf(z)=\frac{z}{\left(1-l_u(z)^2\right)^{1-\alpha}},\quad z\in\mathbb{B},\quad \|u\|=1.
 \end{align*}
 
 According to Lemma 2.2 \cite{Xu-He-Xu}, it is clear that $F\in\mathcal{S}^*(\mathbb{B})$. Note that 
 \begin{align*}
 	f(z)=\sum_{n=0}^{\infty}\frac{(1-\alpha)(2-\alpha)\ldots (n-\alpha)}{n!}(l_u(z))^{2n},\quad z\in\mathbb{B}.
 \end{align*}
 
 Since $F(z)=zf(z)$, we have (see \cite[Theorem 2]{Xiong-Yu-Wang-2025})
 \begin{align}\label{Th-2x}
 	D^mF(0)(z^m)=mD^{m-1}f(0)(z^{m-1})z,
 \end{align}
 where $m\in\mathbb{N}$.
 Consequently from \eqref{Th-2x} it follows that
 \begin{align}\label{Th-2y}
 	\displaystyle
 	\frac{l_{z}(D^mF(0)(z^m))}{m!}=\frac{D^{m-1}f(0)(z^{m-1})}{(m-1)!}||z||.
 \end{align}
 
 Clearly \eqref{Th-2y} yields
 \begin{align*}
 	\frac{l_z(D^2F(0)(z^2))}{2!}=0,\;\;
 	\frac{l_z(D^3F(0)(z^3))}{3!}=(1-\alpha)(l_u(z))^2\|z\|\;\;\text{and}\;\;
 	\frac{l_z(D^4F(0)(z^4))}{4!}=0.
 \end{align*}
 
 Taking $z=ru\;(0<r<1)$, we get
 \begin{align*}
 	\frac{l_z(D^2F(0)(z^2))}{2!||z||^2}=0,\;\;
 	\frac{l_z(D^3F(0)(z^3))}{3!||z||^3}=(1-\alpha)\; \text{and} \;
 	\frac{l_z(D^4F(0)(z^4))}{4!||z||^4}=0.
 \end{align*}
 
 Therefore
 \begin{align*}
 	\left|\frac{l_{z}(D^2F(0)(z^2))}{2!\|z\|^2}\frac{l_{z}(D^4F(0)(z^4))}{4!\|z\|^4}-\left(\frac{l_z(D^3F(0)(z^3))}{3!\|z\|^3}\right)^2\right|=(1-\alpha)^2.
 \end{align*}
\begin{theo}\label{Th-22}Let $f\in\mathcal{H}(\mathbb{U}^n,\mathbb{C})$ with $f(0)=1$ and $F(z)=zf(z)\in\mathcal{S}^*_{\alpha}(\mathbb{U}^n)$ and let $\mu\in[0,1]$. Then
 \begin{align*}
 \displaystyle
 &\left\|\frac{1}{3!}D^3F(0)\left(z^2,\frac{D^3F(0)(z^3)}{3!}\right)-\mu\frac{1}{4!}D^4F(0)\left(z^3,\frac{D^2F(0)(z^2)}{2!}\right)\right\|_{\infty}\\\leq&
 {(1-\alpha)^2}(9-8\mu)||z||_{\infty}^6\nonumber
 \end{align*}
 for all $z\in\mathbb{U}^n$.
 \end{theo}
 
 From Theorem \ref{Th-22}, we get the following corollary.
 \begin{cor}\label{Cr-22}Let $f\in\mathcal{H}(\mathbb{U}^n,\mathbb{C})$ with $f(0)=1$ and $F(z)=zf(z)\in\mathcal{S}^*_{\alpha}(\mathbb{U}^n)$ and let $\mu=1$. Then
 \begin{align*}
 \displaystyle
 &\left\|\frac{1}{3!}D^3F(0)\left(z^2,\frac{D^3F(0)(z^3)}{3!}\right)-\frac{1}{4!}D^4F(0)\left(z^3,\frac{D^2F(0)(z^2)}{2!}\right)\right\|_{\infty}\\\leq&
 (1-\alpha)^2||z||_{\infty}^6\nonumber
 \end{align*}
 for all $z\in\mathbb{U}^n$. The bound is sharp.
 \end{cor}
 
 To establish the sharpness of the bound of Corollary \ref{Cr-22}, let us consider the function $F$ defined by $F(z)=zf(z)=\left(\frac{z_1}{(1-z_1^2)^{1-\alpha}},\frac{z_2}{(1-z_1^2)^{1-\alpha}},\ldots,\frac{z_n}{(1-z_1^2)^{1-\alpha}}\right)$.
 
 It is clear that $F\in\mathcal{S}^*_{\alpha}(\mathbb{U}^n)$. 
  Note that 
 \begin{align*}
 	F(z)=z\sum_{n=0}^{\infty}\frac{(1-\alpha)(2-\alpha)\ldots (n-\alpha)}{n!}z_1^{2n},\quad z\in\mathbb{U}^n.
 \end{align*}
 
 Clearly \eqref{Th-2x} yields
 \begin{align*}
 	\frac{D^2F(0)(z^2)}{2!}=0,\;\;
 	\frac{D^3F(0)(z^3)}{3!}=(1-\alpha)z_1^2 z\;\;\text{and}\;\;
 	\frac{D^4F(0)(z^4)}{4!}=0.
 \end{align*}

 Taking $z=(r,0,\ldots,0)$, where $0<r<1$; one can easily verify that 
 \begin{align*}
 \displaystyle
 \left\|\frac{1}{3!}D^3F(0)\left(z^2,\frac{D^3F(0)(z^3)}{3!}\right)-\frac{1}{4!}D^4F(0)\left(z^3,\frac{D^2F(0)(z^2)}{2!}\right)\right\|_{\infty}=
 (1-\alpha)^2||z||_{\infty}^6,
 \end{align*}
 which establishes the sharpness of the bound.

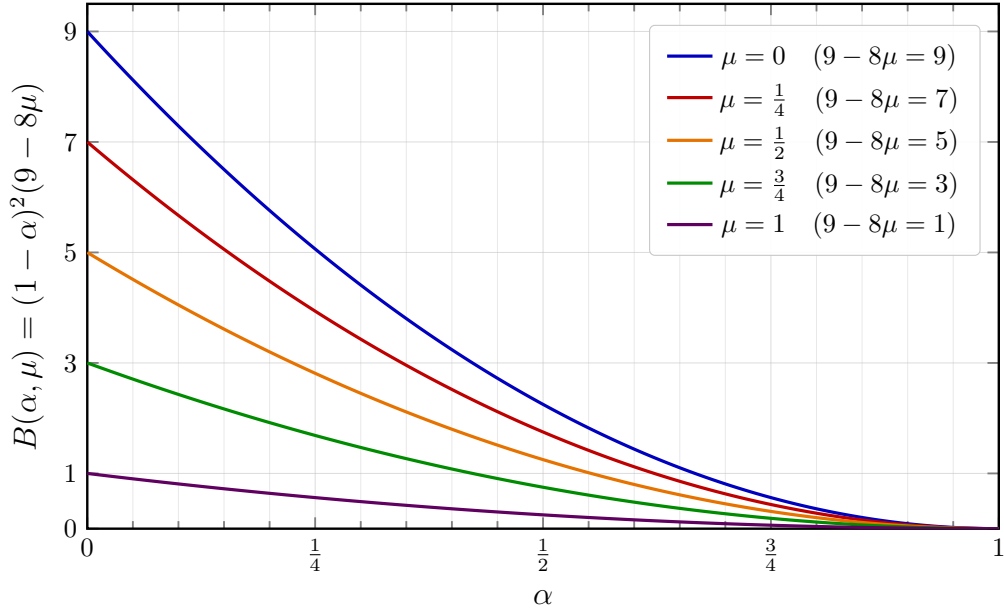
\begin{figure}[htbp]
	\centering
	\begin{tikzpicture}
		\begin{axis}[
			width=0.88\textwidth,
			height=0.55\textwidth,
			xmin=0,
			xmax=1,
			ymin=0,
			ymax=9.5,
			xlabel={$\alpha$},
			ylabel={$B(\alpha,\mu)=(1-\alpha)^2(9-8\mu)$},
			xtick={0,0.25,0.5,0.75,1},
			ytick={0,1,3,5,7,9},
			xticklabels={
				$0$,
				$\frac14$,
				$\frac12$,
				$\frac34$,
				$1$
			},
			minor x tick num=4,
			minor y tick num=1,
			xticklabel style={font=\small},
			yticklabel style={font=\small},
			xlabel style={font=\large},
			ylabel style={font=\large},
			grid=both,
			major grid style={
				line width=0.35pt,
				opacity=0.55
			},
			minor grid style={
				line width=0.2pt,
				opacity=0.35
			},
			axis line style={line width=0.9pt},
			tick style={line width=0.7pt},
			legend style={
				font=\small,
				draw=gray!45,
				fill=white,
				fill opacity=0.94,
				text opacity=1,
				at={(0.98,0.96)},
				anchor=north east,
				rounded corners=2pt,
				inner xsep=6pt,
				inner ysep=5pt,
				row sep=2pt
			},
			domain=0:1,
			samples=150
			]
			
			\addplot[
			very thick,
			smooth,
			blue!75!black
			]
			{9*(1-x)^2};
			\addlegendentry{$\mu=0\quad (9-8\mu=9)$}
			
			\addplot[
			very thick,
			smooth,
			red!75!black
			]
			{7*(1-x)^2};
			\addlegendentry{$\mu=\frac14\quad (9-8\mu=7)$}
			
			\addplot[
			very thick,
			smooth,
			orange!90!black
			]
			{5*(1-x)^2};
			\addlegendentry{$\mu=\frac12\quad (9-8\mu=5)$}
			
			\addplot[
			very thick,
			smooth,
			green!55!black
			]
			{3*(1-x)^2};
			\addlegendentry{$\mu=\frac34\quad (9-8\mu=3)$}
			
			\addplot[
			very thick,
			smooth,
			violet!75!black
			]
			{(1-x)^2};
			\addlegendentry{$\mu=1\quad (9-8\mu=1)$}
			
		\end{axis}
	\end{tikzpicture}
	
	\caption{The sharp bound $(1-\alpha)^2(9-8\mu)$ for $\mu=0,\frac14,\frac12,\frac34,1$ at $\alpha=0$. The bound reduces to $9-8\mu$, yielding the respective values $9,7,5,3,1$. The figure illustrates the monotone decrease as $\mu$ increases, with the minimum value $1$ attained at $\mu=1$.}
	\label{fig:bound_comparison}
\end{figure}
 \begin{rem} In the special case where $X = \mathbb{C}$ and $\mathbb{B} = \mathbb{U}$, Theorems \emph{\ref{Th-21}}--\emph{\ref{Th-22}} reduce to Theorem A.
 \end{rem}

 \subsection{{Auxiliary lemma}}
 Let $\mathcal{P}$ denote the class of analytic functions $h$ in the open unit disk $\mathbb{U}$ satisfying $h(0) = 1$ and $\Re(h(z)) > 0$ for all $z \in \mathbb{U}$. Each $h\in\mathcal{P}$ is called a Carath\'eodory function. Coefficient bounds for functions in $\mathcal{P}$ play a fundamental role in solving various coefficient problems. Let $h(z)=1+\sum_{n=1}^{\infty}h_nz^n\in\mathcal{P}$. Then we know that
 $|h_n|\leq 2$ and $|h_n-h_mh_{n-m}|\leq 2,\quad n\geq 2,\;1\leq m\leq n-1$ (see \cite{Cho-Kumar-Ravichandran-2019}).
 
 \begin{lem}\label{lm1}\emph{\cite{Libera_1,Libera_2}} If $h\in\mathcal{P}$ is of the form $h(z)=1+\sum_{n=1}^{\infty}h_nz^n$ with $h_1\geq 0$, then
 \[
 h_2=h_1^2+x(4-h_1^2)
 \]
 and
 \[
 4h_3=h_1^3+2h_1(4-h_1^2)x-h_1(4-h_1^2)x^2+2(4-h_1^2)(1-|x|^2)\eta
 \]
 for some $x$ and $\eta$ such that $|x|\leq 1$ and $|\eta|\leq 1$.
 \end{lem}

\begin{proof}[\bf Proof of Theorem \ref{Th-21}]
By the given condition, we have $F\in \mathcal{S}^*_{\alpha}(\mathbb{B})$, where $0\leq \alpha<1$. Then \eqref{SL1} yields
\begin{align*}
\Re\left\lbrace \frac{\|z\|}{l_z\left([D F(z)]^{-1}F(z)\right)}\right\rbrace>\alpha, \quad z\in\mathbb{B}\backslash \{0\},\quad l_z\in T(z).
\end{align*}

Let $z\in \mathbb{B}\backslash\{0\}$ be fixed and let $z_0=z/\|z\|$. The function $g:\mathbb{U}\to \mathbb{C}$ is defined as follows
\begin{align}\label{Th-21.1}	
g(\zeta)=\begin{cases}
\frac{\zeta}{l_{z_0}\left([D F(\zeta z_0)]^{-1}F(\zeta z_0)\right)},& \zeta\neq 0,\\[8pt]
1& \zeta=0.
\end{cases}
\end{align}
Clearly $g\in\mathcal{H}(\mathbb{U})$ and $g(0)=1$.
Since $l_{\zeta z}(.)=\frac{|\zeta|}{\zeta}l_z$, for each $\zeta\in\mathbb{C}\backslash\{0\}$, we have
\begin{align*}
g(\zeta)=\frac{\zeta}{l_{z_0}\left([D F(\zeta z_0)]^{-1}F(\zeta z_0)\right)}=\frac{\|\zeta z_0\|}{l_{\zeta z_0}\left([D F(\zeta z_0)]^{-1}F(\zeta z_0)\right)}
\end{align*}	
for $\zeta\in\mathbb{U}\backslash\{0\}$ and so  $\Re\{g(\zeta)\}>\alpha$
for all $\zeta\in\mathbb{U}$. 
\par Let $h(\zeta)=\frac{g(\zeta)-\alpha}{1-\alpha}$ for all $\zeta\in\mathbb{U}$. Clearly $h\in\mathcal{H}(\mathbb{U})$ with $\Re(h(\zeta))>0$ for all $\zeta\in\mathbb{U}$ and $h(0)=1$. Also let
$h(\zeta)=1+h_1\zeta+h_2\zeta^2+h_3\zeta^3+\ldots$. Then it follows that,
\begin{align}\label{Th-21.2}
g(\zeta)=1+(1-\alpha)h_1\zeta+(1-\alpha)h_2\zeta^2+(1-\alpha)h_3\zeta^3+\ldots
\end{align}	

Using a similar method as in Theorem 7.1.14 \cite{Graham-Kohr}, we have
\begin{align}\label{Th-21.2b}
[D F(z)]^{-1}\eta=\frac{1}{f(z)}\left\lbrack \eta-\frac{\frac{zDf(z)}{f(z)}}{1+\frac{Df(z)z \eta}{f(z)}}\right\rbrack,\quad \eta\in X.
\end{align}

Therefore 
\begin{align}\label{Th-21.3}
	[D F(z)]^{-1}F(z)=\frac{zf(z)}{f(z)+D f(z)z},
\end{align}
which implies that $\frac{\|z\|}{l_z\left([D F(z)]^{-1}F(z)\right)}=1+\frac{Df(z) z}{f(z)}$
and so
\begin{align}\label{Th-21.4}
g(\zeta)=\frac{\|\zeta z_0\|}{l_{\zeta z_0}\left([D F(\zeta z_0)]^{-1}F(\zeta z_0)\right)}=1+\frac{Df(\zeta z_0) \zeta z_0}{f(\zeta z_0)}.
\end{align}

Let
\begin{align}\label{Th-21.5}
f(\zeta z_0)=1+D f(0)(z_0)\zeta+\frac{1}{2!}D^2 f(0)(z_0^2)\zeta^2+\frac{1}{3!}D^3 f(0)(z_0^3)\zeta^3+\ldots.
\end{align}

Clearly
\begin{align}\label{Th-21.6}
D f(\zeta z_0)(\zeta z_0)=D f(0)(z_0)\zeta+D^2 f(0)(z_0^2)\zeta^2+\frac{1}{2}D^3 f(0)(z_0^3)\zeta^3+\ldots
\end{align}

Now using \eqref{Th-21.5} and \eqref{Th-21.6} to \eqref{Th-21.4}, we get
\begin{align}\label{Th-21.7}
&\left(1+(1-\alpha)h_1\zeta+(1-\alpha)h_2\zeta^2+(1-\alpha)h_3\zeta^3+\ldots\right)\times\\&\left(1+D f(0)(z_0)\zeta+\frac{1}{2!}D^2 f(0)(z_0^2)\zeta^2+\frac{1}{3!}D^3 f(0)(z_0^3)\zeta^3+\ldots\right)\nonumber\\=&
1+2Df(0)(z_0)\zeta+\frac{3}{2}D^2f(0)(z_0^2)\zeta^2+\frac{2}{3}D^3f(0)(z_0^3)\zeta^3+\ldots.\nonumber
\end{align}

Clearly from \eqref{Th-21.7}, we find that
\begin{align}\label{Th-21.8}
\begin{cases}
Df(0)(z_0)&=(1-\alpha)h_1,\\[8pt]
D^2f(0)(z_0^2)&=(1-\alpha)h_2+((1-\alpha)h_1)^2,\\[8pt]
D^3f(0)(z_0^3)&=\left((1-\alpha)h_1\right)^3+3(1-\alpha)^2h_1h_2+2(1-\alpha)h_3.
\end{cases}
\end{align}

Using \eqref{Th-2y} to \eqref{Th-21.8}, we get
\begin{align}\label{Th-21.9}
&\frac{l_{z}(D^4F(0)(z^4))}{4!\|z\|^4}=\frac{1}{6}\left(\left((1-\alpha)h_1\right)^3+3(1-\alpha)^2h_1h_2+2(1-\alpha)h_3\right),
\end{align}
\begin{align}\label{Th-21.10}
\frac{l_{z}(D^3F(0)(z^3))}{3!\|z\|^3}=\frac{1}{2}\left((1-\alpha)h_2+((1-\alpha)h_1)^2\right)
\end{align}
 and
 \begin{align}\label{Th-21.11}
 \frac{l_z(D^2F(0)(z^2))}{2!\|z\|^2}=(1-\alpha)h_1.
\end{align}

Now using \eqref{Th-21.9}-\eqref{Th-21.11}, we find that 
\begin{align*}
	&\left|\mu\frac{l_{z}(D^2F(0)(z^2))}{2!\|z\|^2}\frac{l_{z}(D^4F(0)(z^4))}{4!\|z\|^4}-\left(\frac{l_z(D^3F(0)(z^3))}{3!\|z\|^3}\right)^2\right|\\=&
	\frac{(1-\alpha)^2}{12}\left|4\mu h_1h_3-3h_2^2+6(\mu-1)(1-\alpha)h_1^2h_2+(2\mu-3)(1-\alpha)^2h_1^4\right|
\end{align*}
and so using Lemma \ref{lm1}, we get
\begin{align}\label{Th-21.12}
	&\left|\mu\frac{l_{z}(D^2F(0)(z^2))}{2!\|z\|^2}\frac{l_{z}(D^4F(0)(z^4))}{4!\|z\|^4}-\left(\frac{l_z(D^3F(0)(z^3))}{3!\|z\|^3}\right)^2\right|\\\le&\frac{(1-\alpha)^2}{48}\Biggl|\bigl((4\mu-3)+12(\mu-1)(1-\alpha)+4(2\mu-3)(1-\alpha)^2\bigr)(h_1)^4\nonumber\\&+(4-(h_1)^2)\left\lbrace (8\mu-6+12(\mu-1)(1-\alpha))h_1^2x-((4\mu-3)h_1^2+12)x^2\nonumber\right.\\&\left.
	+8\mu h_1(1-|x|^2)\eta\right\rbrace\Biggr|\nonumber
\end{align}
for some complex numbers \(x\) and \(\eta\) satisfying $|x|\le1$ and $|\eta |\le 1$.
Since the class $\mathcal{P}$ is invariant under rotations, we may assume, $h_1 =t\in[0, 2]$, $x_1 := |x|$ and $y := |\eta|$.

Clearly from (\ref{Th-21.12}), we get
\begin{align}\label{Th-21.13}
&\left|\mu\frac{l_{z}(D^2F(0)(z^2))}{2!\|z\|^2}\frac{l_{z}(D^4F(0)(z^4))}{4!\|z\|^4}-\left(\frac{l_z(D^3F(0)(z^3))}{3!\|z\|^3}\right)^2\right|\\&\le\frac{(1-\alpha)^2}{48}\biggl|\tilde{F}(t,x_1,y)\biggr|,\nonumber
\end{align}
where 
\begin{align}
&\tilde{F}(t,x_1,y):\\=&|A|t^4+(4-t^2)\left((8\mu-6+12(|B|t^2x_1+|(4\mu-3)t^2+12|x_1^2
+8\mu t(1-x_1^2)y\right).\nonumber
\end{align}
with $A=(4\mu-3)+12(\mu-1)(1-\alpha)+4(2\mu-3)(1-\alpha)^2$ and $B=8\mu-6+12(\mu-1)(1-\alpha)$.
But we can deduce that
\begin{align}\label{Th.21.14}
	\tilde{F}(t,x_1,y)&\le \tilde{F}(t,x_1,1)\\&=|A|t^4+(4-t^2)\left(|B|t^2x_1+|(4\mu-3)t^2+12|x_1^2
	+8\mu t(1-x_1^2)\right)\nonumber\\&:=M(t,x_1).\nonumber
\end{align}
Note that
\begin{align*}
	|A|\le |24\mu-27|=27-24\mu, \qquad \mu\in[0,1]
\end{align*}
and
\begin{align}\label{S1}
	|B|=\begin{cases}
	\displaystyle
	18-20\mu,& 0\le\mu\le \frac{9}{10},\\[2ex]
	\displaystyle
	20\mu-18,& \frac{9}{10}<\mu\le1.
\end{cases}
\end{align}

Also $(4\mu-3)t^2+12\geq 0$ for every $t\in [0,2]$.
Now from (\ref{Th.21.14}), we have 
\begin{align*}
	M(t,x_1)=(27-24\mu)t^4+(4-t^2)\left(|B|t^2x_1+((4\mu-3)t^2+12)x_1^2
	+8\mu t(1-x_1^2)\right),
\end{align*}
and so
\begin{align*}
	\frac{\partial M}{\partial x_1}= (4-t^2)\biggl(|B|t^2+2((4\mu-3)t^2+12-8\mu t)\biggr)x_1.
\end{align*}
Since $|B|>0$ and $(4\mu-3)t^2+12-8\mu t>0$, $x_1=1$ gives the maximum of $M(t,x_1)$.
Therefore
\begin{align}\label{s2}
	M(t,1)=(27-24\mu)t^4+(4-t^2)((|B|+4\mu-3)t^2+12).
\end{align}

If $0\le\mu\le\frac{9}{10}$, then from \eqref{S1} and \eqref{s2} we have
\begin{align*}
M(t,1)=(12-8\mu)t^4+16(3-4\mu)t^2+48.
\end{align*}
Also this can be rewritten as $M(t^2,1):=M_1(u)=(12-8\mu)u^2+16(3-4\mu)u+48$,
where $u=t^2$. Clearly $12-8\mu>0$ for any $\mu\in[0,1]$. Therefore, the quadratic function is convex and so the maximum occurs at the end point. Now $M_1(0)=48$ and $M_1(4)= 432-384\mu$. Hence, from \eqref{s2} we obtain \begin{align*}
	M(t,1)\le 48(9-8\mu).
\end{align*}

If $\frac{9}{10}< \mu\le 1$, then from \eqref{S1} and \eqref{s2} we obtain
\begin{align*}
M(t,1)=48\left((1-\mu)t^4-2(1-\mu)t^2+1\right).
\end{align*}
Since $t\in[0, 2]$. Now it is easy to deduce that 
  \begin{align*}
  	M(t,1)\le48\left(9-8\mu\right).
  \end{align*}
Therefore, collecting all the scenarios, from \eqref{Th.21.14} we obtain the desired bound 
\begin{align*}
\tilde{F}(t,x_1,y)\le 48\left(9-8\mu\right).
\end{align*}

Hence, from \eqref{Th-21.13} we get 
\begin{align*}
	\left|\mu\frac{l_{z}(D^2F(0)(z^2))}{2!\|z\|^2}\frac{l_{z}(D^4F(0)(z^4))}{4!\|z\|^4}-\left(\frac{l_z(D^3F(0)(z^3))}{3!\|z\|^3}\right)^2\right|\le{(1-\alpha)^2}(9-8\mu).
\end{align*}
\end{proof}

\begin{proof}[\bf Proof of Theorem \ref{Th-22}]
By the given condition, $F\in\mathcal{S}^*_{\alpha}(\mathbb{U}^n)$, where $0\leq \alpha<1$. Therefore
\begin{align}\label{Th-22.1}
\displaystyle
\Re\left\lbrace \frac{\|z\|_{\infty}}{l_z\left([D F(z)]^{-1}F(z)\right)}\right\rbrace>\alpha, \quad z\in\mathbb{U}^n\backslash \{0\},\quad l_z\in T(z).
\end{align}
	
Let $w(z)=(w_1(z),w_2(z),\ldots,w_n(z))'=[D F(z)]^{-1}F(z)$ be a column vector in $\mathbb{C}^n$. According to (\ref{eq3}), we have
\begin{align}\label{Th-22.2}
\displaystyle
l_z(w)=\sum\limits_{\{j:||z||_{\infty}=|z_j|\}}t_j\frac{w_j(z)\ol z_j}{||z||_{\infty}},
\end{align}
where each $t_j\geq 0$ and $\sum\limits_{\{j:||z||_{\infty}=|z_j|\}}t_j=1.$
\par Let $k\in\{1,2,\ldots,n\}$ be such that $||z||_{\infty}=|z_k|$. Suppose $t_k=1$ and $t_j=0$, where $j\in\{1,2,\ldots,n\}$ such that $j\neq k$. Then from \eqref{Th-22.1} and \eqref{Th-22.2}, we have
\begin{align}\label{Th-22.3}
\displaystyle
\Re\left\lbrace \frac{z_k}{w_k(z)}\right\rbrace>\alpha,\quad z\in\mathbb{U}^n\backslash\{0\},
\end{align}
where $w_k(z)$ is the $k$-th coordinate of $w(z)$ such that $||z||_{\infty}=|z_k|$.
	
Let $z\in\mathbb{U}^n\backslash\{0\}$ and let $z_0=z/\|z\|$. We define the function $g_k:\mathbb{U}\to \mathbb{C}$ as follows
\begin{align}\label{Th-22.4}	
\displaystyle
g_k(\zeta)=\begin{cases}
\frac{\zeta z_k}{w_k(\zeta z_0)},& \zeta\neq 0,\\[8pt]
1& \zeta=0.
\end{cases}
\end{align}
	
Now it follows from \eqref{Th-22.3} and \eqref{Th-22.4} that $\Re\{g_k(\zeta)\}>0$ for all $\zeta\in\mathbb{U}$. 
	
Let
\begin{align}\label{Th-22.5}
\displaystyle
h_k(\zeta)=\frac{g_k(\zeta)-\alpha}{1-\alpha}
\end{align}
for all $\zeta\in\mathbb{U}$. Since $\Re\{g_k(\zeta)\}>0$ 
for all $\zeta\in\mathbb{U}$, we deduce from \eqref{Th-22.5} that $h_k\in\mathcal{H}(\mathbb{U})$ with $\Re\{h_k(\zeta)\}>0$ for all $\zeta\in\mathbb{U}$ and $h_k(0)=1$.

On the other hand, from \eqref{Th-21.3}, we have
\begin{align}\label{Th-22.6}
\displaystyle
w_k(\zeta z_0)=\zeta z_k\left\lbrack \frac{f(\zeta z_0)}{f(\zeta z_0)+D f(\zeta z_0)\zeta z_0}\right\rbrack.
\end{align}
	
Using \eqref{Th-22.4} and \eqref{Th-22.6} to \eqref{Th-22.5}, we get
\begin{align}\label{Th-22.7}
\displaystyle
\left((1-\alpha)h_k(z)+\alpha\right)f(\zeta z_0)=f(\zeta z_0)+D f(\zeta z_0)\zeta z_0.	
\end{align} 	
	
Let
\begin{align}\label{Th-22.7a}
h_k(\zeta)=1+h_{k1}\zeta+h_{k2}\zeta^2+h_{k3}\zeta^3+\ldots
\end{align}
	
Now using \eqref{Th-21.5}, \eqref{Th-21.6} and \eqref{Th-22.7a} to \eqref{Th-22.7} and then comparing coefficients, we obtain
\begin{align}\label{Th-22.8}
\begin{cases}
Df(0)(z_0)&=(1-\alpha)h_{k1},\\[8pt]
D^2f(0)(z_0^2)&=(1-\alpha)h_{k2}+((1-\alpha)h_{k1})^2,\\[8pt]
D^3f(0)(z_0^3)&=\left((1-\alpha)h_{k1}\right)^3+3(1-\alpha)^2h_{k1}h_{k2}+2(1-\alpha)h_{k3}.
\end{cases}
\end{align}

Let $F(z)=(F_1(z),F_2(z),\ldots,F_n(z))$. Since $F(z)=zf(z)$, it follows that
\begin{align*}
	\displaystyle
	\frac{D^mF_k(0)(z^m)}{m!}=\frac{D^{m-1}f(0)(z^{m-1})}{(m-1)!}z_k.
\end{align*}

Since $z_0=z/\|z\|$, we have 
\begin{align}\label{Th-22.11}
	\displaystyle
	\frac{D^mF_k(0)(z_0^m)}{m!}=\frac{D^{m-1}f(0)(z_0^{m-1})}{(m-1)!}\frac{z_k}{||z||_{\infty}}.
\end{align}

In view of \eqref{Th-2x} and \eqref{Th-22.11}, we obtain
\begin{align}\label{Th-22.12}
\displaystyle
&\left|\frac{1}{3!}D^3F_k(0)\left(z_0^2,\frac{D^3F(0)(z_0^3)}{3!}\right)\frac{||z||_{\infty}}{z_k}-\mu\frac{1}{4!}D^4F_k(0)\left(z_0^3,\frac{D^2F(0)(z_0^2)}{2!}\right)\frac{||z||_{\infty}}{z_k}\right|\nonumber\\=&
\left|\frac{1}{3!}D^3F_k(0)\left(z_0^2,\frac{D^2f(0)(z_0^2)z_0}{2!}\right)\frac{||z||_{\infty}}{z_k}-\mu \frac{1}{4!}D^4F_k(0)\left(z_0^3,Df(0)(z_0)z_0\right)\frac{||z||_{\infty}}{z_k}\right|\nonumber\\=&
\left|\frac{1}{2!} D^2f(0)(z_0^2)\frac{1}{3!}D^3F_k(0)\left(z_0^2,z_0\right)\frac{||z||_{\infty}}{z_k}-\mu Df(0)(z_0)\frac{1}{4!}D^4F_k(0)(z_0^4)\frac{||z||_{\infty}}{z_k}\right|
\nonumber\\=&
\left| \frac{1}{2!}D^2f(0)(z_0^2)\frac{1}{3!}D^3F_k(0)\left(z_0^3\right)\frac{||z||_{\infty}}{z_k}-\mu Df(0)(z_0)\frac{D^3f(0)(z_0^3)}{3!}\right|
\nonumber\\=&
\left|\left(\frac{D^2f(0)(z_0^2)}{2!}\right)^2-\mu Df(0)(z_0)\frac{D^3f(0)(z_0^3)}{3!}\right|.
\end{align}

Using \eqref{Th-21.8} to \eqref{Th-22.12}, we obtain
\begin{align*}
\displaystyle
&\left|\frac{1}{3!}D^3F_k(0)\left(z_0^2,\frac{D^3F(0)(z_0^3)}{3!}\right)\frac{||z||_{\infty}}{z_k}-\mu\frac{1}{4!}D^4F_k(0)\left(z_0^3,\frac{D^2F(0)(z_0^2)}{2!}\right)\frac{||z||_{\infty}}{z_k}\right|\\=&
\frac{(1-\alpha)^2}{12}\left|4\mu h_{k1}h_{k3}-3h_{k2}^2+6(\mu-1)(1-\alpha)h_{k1}^2h_{k2}+(2\mu-3)(1-\alpha)^2h_{k1}^4\right|.	
	\nonumber
\end{align*}

Using similar arguments as in the proof of Theorem \ref{Th-21}, we have
\begin{align*}
\displaystyle
&\left|\frac{1}{3!}D^3F_k(0)\left(z_0^2,\frac{D^3F(0)(z_0^3)}{3!}\right)\frac{||z||_{\infty}}{z_k}-\mu\frac{1}{4!}D^4F_k(0)\left(z_0^3,\frac{D^2F(0)(z_0^2)}{2!}\right)\frac{||z||_{\infty}}{z_k}\right|\\&\le{(1-\alpha)^2}(9-8\mu).
\nonumber
\end{align*}
In particular, when $z_0\in\partial_0 \mathbb{U}^n$, we deduce that
\begin{align}\label{Th-22.13}
\displaystyle
&\left|\frac{1}{3!}D^3F_k(0)\left(z_0^2,\frac{D^3F(0)(z_0^3)}{3!}\right)-\mu\frac{1}{4!}D^4F_k(0)\left(z_0^3,\frac{D^2F(0)(z_0^2)}{2!}\right)\right|\\\leq& 
{(1-\alpha)^2}(9-8\mu),	\quad k=1,2,\ldots,n.\nonumber
\end{align}

We know that
\begin{align*}
\frac{1}{3!}D^3F_k(0)\left(z^2,\frac{D^3F(0)(z^3)}{3!}\right)-\mu\frac{1}{4!}D^4F_k(0)\left(z^3,\frac{D^2F(0)(z^2)}{2!}\right)\in\mathcal{H}(\ol{\mathbb{U}^n}),
\end{align*}
$k=1,2,\ldots,n$.
By the maximum modulus theorem for holomorphic functions on the unit polydisk, we deduce from \eqref{Th-22.13} that
\begin{align*}
\displaystyle
\left|\frac{1}{3!}D^3F_k(0)\left(z_0^2,\frac{D^3F(0)(z_0^3)}{3!}\right)-\mu\frac{1}{4!}D^4F_k(0)\left(z_0^3,\frac{D^2F(0)(z_0^2)}{2!}\right)\right|\leq
{(1-\alpha)^2}(9-8\mu)
\end{align*}
for all $z_0\in\partial \mathbb{U}^n$, where $k=1,2,\ldots,n$.
Hence
\begin{align*}
\displaystyle
\left|\frac{1}{3!}D^3F_k(0)\left(z^2,\frac{D^3F(0)(z^3)}{3!}\right)-\mu\frac{1}{4!}D^4F_k(0)\left(z^3,\frac{D^2F(0)(z^2)}{2!}\right)\right|\leq {(1-\alpha)^2}(9-8\mu)
\end{align*}
for all $z\in\mathbb{U}^n$, where $k=1,2,\ldots,n$. Therefore
\begin{align*}
\displaystyle
&\left\|\frac{1}{3!}D^3F(0)\left(z^2,\frac{D^3F(0)(z^3)}{3!}\right)-\mu\frac{1}{4!}D^4F(0)\left(z^3,\frac{D^2F(0)(z^2)}{2!}\right)\right\|_{\infty}\\\leq&
{(1-\alpha)^2}(9-8\mu)||z||_{\infty}^6, \quad z\in\mathbb{U}^n.
\end{align*}
\end{proof}

\section{\bf Bounds for $H_{2,2}^{\mu}(f)$ over the class of quasi-convex mappings of type $B$ in complex Banach spaces}\label{Sec-4}
Let $X$ and $Y$ be two complex Banach spaces and $\Omega\subset X$ be a domain.
Also let $f:\Omega\rightarrow Y$ be a holomorphic mapping. We say that $f$ is
\emph{convex} on $\Omega$ if $f$ is biholomorphic on $\Omega$ and
$f(\Omega)$ is a convex domain in $Y$.\vspace{1.2mm}

A normalized locally biholomorphic mapping $f\in\mathcal{H}(\mathbb{B})$ is said to be quasi-convex of type $B$ if (see \cite[Definition 2.6]{Ham2})
\begin{align}\label{C1}
	\displaystyle
	\Re\left\lbrace l_z\left([Df(z)]^{-1}\left(D^2f(z)(z,z)+Df(z)z\right)\right)\right\rbrace>0,\quad z\in\mathbb{B}.
\end{align}

If $\mathbb{B}$ is the unit ball of $X$, we let $\mathcal{K}_B(\mathbb{B})$ denote the subset of $\mathcal{S}(\mathbb{B})$
consisting of normalized quasi-convex mappings of type $B$ from $\mathbb{B}$ into $X$.\vspace{1.2mm}

Convex mappings in higher dimensions have been extensively studied in recent literature. Comprehensive overviews can be found in the monographs by Gong \cite{Gong} and Kohr \& Liczberski  \cite{Kohr-Liczberski-1998}. For foundational work on this topic, see \cite{Ham4},
\cite{Rop-Su}-\cite{Su3}, \cite{Xu-Dl-Li}, \cite{Xu-Yang-Liu-Xu-2015}, \cite{Zhang-Liu-2002}.
\vspace{1.2mm}

In this section, we extend Theorem B to the case of complex Banach space.
\begin{theo}\label{Th-31}Let $f\in\mathcal{H}(\mathbb{B},\mathbb{C})$ with $f(0)=1$ and $F(z)=zf(z)\in\mathcal{K}_B(\mathbb{B})$ and let $\mu\in [0,1]$. Then
\begin{align*}
\left|\mu\frac{l_{z}(D^2F(0)(z^2))}{2!\|z\|^2}\frac{l_{z}(D^4F(0)(z^4))}{4!\|z\|^4}-\left(\frac{l_z(D^3F(0)(z^3))}{3!\|z\|^3}\right)^2\right|\le\max\left\{\frac{1}{9}, \frac{\mu}{8}, 1-\mu\right\}.
\end{align*}
\end{theo}
\begin{figure}[H]
	\centering
	\begin{tikzpicture}
		\begin{axis}[
			title={Comparison of bounds in Theorem \ref{Th-31}},
			xlabel={$\mu$},
			ylabel={Value},
			xmin=0,
			xmax=1,
			ymin=0,
			ymax=1.05,
			grid=both,
			grid style={line width=.1pt, draw=gray!30},
			major grid style={line width=.2pt, draw=gray!60},
			minor tick num=4,
			legend style={
				at={(0.97,0.97)},
				anchor=north east,
				font=\small\bfseries,
				draw=black,
				fill=white,
				fill opacity=0.9,
				rounded corners=3pt,
				inner sep=5pt,
				legend cell align={left}
			},
			axis line style={thick},
			width=0.8\textwidth,
			height=0.4\textwidth,
			samples=200,
			every axis plot post/.append style={thick},
			xtick={0, 0.2, 0.4, 0.6, 0.8, 0.8889, 1},
			xticklabels={0, 0.2, 0.4, 0.6, 0.8, $\frac{8}{9}$, 1},
			ytick={0, 0.2, 0.4, 0.6, 0.8, 1},
			yticklabels={0, 0.2, 0.4, 0.6, 0.8, 1},
			]
			
			\addplot[
			domain=0:1,
			draw=none,
			fill=blue!8,
			fill opacity=0.3,
			forget plot
			]
			{max(1/9, x/8, 1-x)} \closedcycle;
			
			\addplot[
			domain=0:1,
			color=blue!90!black,
			line width=3.5pt,
			solid
			]
			{1/9};
			\addlegendentry{$1/9$}
			
			\addplot[
			domain=0:1,
			color=red!90!black,
			line width=3.5pt,
			dash pattern=on 8pt off 4pt
			]
			{x/8};
			\addlegendentry{$\mu/8$}
			
			\addplot[
			domain=0:1,
			color=green!60!black,
			line width=3.5pt,
			dash pattern=on 4pt off 3pt
			]
			{1-x};
			\addlegendentry{$1-\mu$}
			
			\addplot[
			domain=0:1,
			color=black,
			line width=4.0pt,
			]
			{max(1/9, x/8, 1-x)};
			\addlegendentry{$\max\{1/9, \mu/8, 1-\mu\}$}

			\addplot[
			only marks,
			mark=*,
			mark size=5pt,
			color=black,
			fill=white,
			draw=black,
			line width=1.5pt
			]
			coordinates {(0.8888889,0.1111111)};
			
			\draw[
			->,
			thick,
			color=black,
			line width=1.5pt
			] (axis cs:0.83,0.332) -- (axis cs:0.8888889,0.1411111);

			\node[
			anchor=south west,
			font=\large\bfseries,
			fill=white,
			draw=black,
			inner sep=5pt,
			rounded corners=4pt,
			line width=1.2pt,
			align=center
			] at (axis cs:0.73,0.35)
			{$(8/9,\,1/9)$};

			\draw[
			densely dashed,
			black!50,
			line width=1pt
			]
			(axis cs:0.8888889,0) -- (axis cs:0.8888889,0.1111111);
			
		\end{axis}
	\end{tikzpicture}
	\caption{Comparison of the three bounds for different values of $\mu$. The overall bound 
		$\max\{1/9,\mu/8,1-\mu\}$ (solid black) equals $1-\mu$ on $[0,8/9]$ and $\mu/8$ on $[8/9,1]$, 
		attaining its minimum value $1/9$ only at $\mu=8/9$ (marked with $\bullet$). 
		The shaded area under the max curve highlights the range of the bound.}
	\label{fig:bound-comparison}
\end{figure}
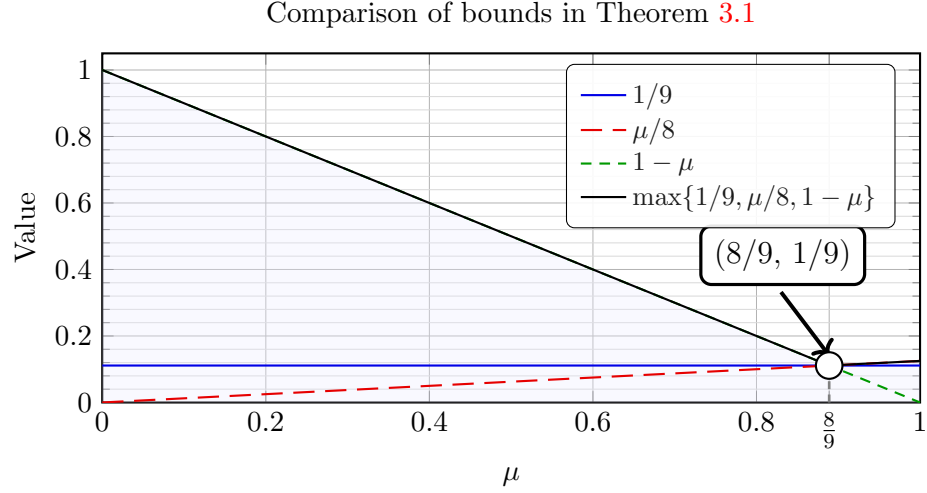
From Theorem \ref{Th-31}, we get the following corollary.
\begin{cor}\label{Cr-31}Let $f\in\mathcal{H}(\mathbb{B},\mathbb{C})$ with $f(0)=1$ and $F(z)=zf(z)\in\mathcal{K}_B(\mathbb{B})$. Then
\begin{align*}
\left|\frac{l_{z}(D^2F(0)(z^2))}{2!\|z\|^2}\frac{l_{z}(D^4F(0)(z^4))}{4!\|z\|^4}-\left(\frac{l_z(D^3F(0)(z^3))}{3!\|z\|^3}\right)^2\right|\le \frac{1}{8}.
\end{align*}
The bound is sharp.
\end{cor}

To establish the sharpness of the bound of Corollary \ref{Cr-31}, let us consider the function $F$ defined by
\begin{align*}
	F(z)=zf(z)=\frac{f_1(l_u(z))}{l_u(z)}z,\quad z\in\mathbb{B},\quad \|u\|=1,
\end{align*}
where 
\begin{align*}
f_1(l_u(z))=\frac{1}{{(\omega^2-\omega)}}\log\frac{1+\omega^2l_u(z)}{1+\omega l_u(z)},
\end{align*}
with $\omega=(1+i\sqrt{3})/2$. Straightforward calculation yields
\begin{align*}
1+\frac{zf_1''(z)}{f_1'(z)}=h(z)=\frac{1-z^2}{1-z+z^2}.
\end{align*}
Note that
\begin{align*}
	\Re\{h(z)\}=\frac{(1-|z|^2)(1+|z|^2-\Re\{z\})}{|1-z+z^2|^2}.
\end{align*}

Clearly $\Re\{h(z)\}>0$ for $z\in\mathbb{U}$ and so $f_1\in\mathcal{K}(\mathbb{U})$. According to Corollary 2 in \cite{Rop-Su}, $F\in\mathcal{K}_B(\mathbb{B})$. On the other hand, we have
\begin{align*}f_1(z)=z+\frac{z^2}{2}-\frac{z^4}{4}-\frac{z^5}{5}+\frac{z^7}{7}+\frac{z^8}{8}+\cdots.
\end{align*}
Therefore 
\begin{align*}
	f(z)=\frac{f_1(l_u(z))}{l_u(z)}=1+\frac{l_u(z)}{2}-\frac{l_u(z)^3}{4}-\frac{l_u(z)^4}{5}+\frac{l_u(z)^6}{7}+\frac{l_u(z)^7}{8}+\cdots
\end{align*}
for $z\in\mathbb{B}$ and $||u||=1$.
Clearly \eqref{Th-2y} yields
\begin{align*}
	\frac{l_z(D^2F(0)(z^2))}{2!}=\frac{l_u(z)||z||}{2},\;\;
	\frac{l_z(D^3F(0)(z^3))}{3!}=0\;\;\text{and}\;\;
	\frac{l_z(D^4F(0)(z^4))}{4!}=\frac{l_u(z)^3||z||}{4}.
\end{align*}

Taking $z=ru\;(0<r<1)$, we get
\begin{align*}
	\frac{l_z(D^2F(0)(z^2))}{2!||z||^2}=\frac{1}{2},\;\;
	\frac{l_z(D^3F(0)(z^3))}{3!||z||^3}=0\;\;\text{and}\;\;
	\frac{l_z(D^4F(0)(z^4))}{4!||z||^4}=\frac{1}{4}.
\end{align*}
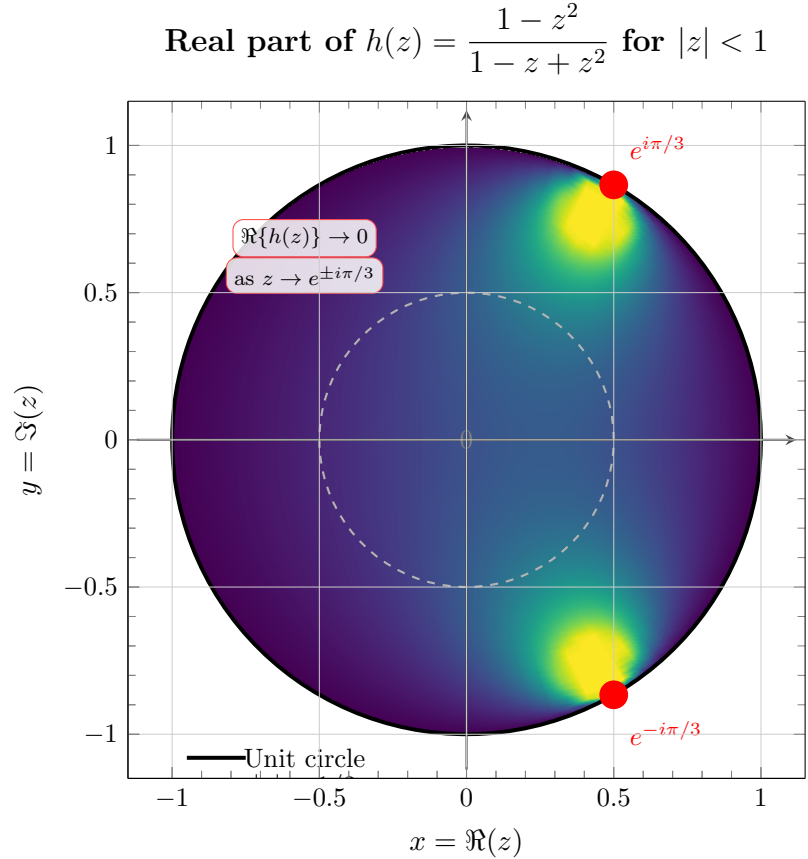
\begin{figure}[H]
	\centering
	\begin{tikzpicture}
		\begin{axis}[
			title={Real part of $h(z)=\dfrac{1-z^2}{1-z+z^2}$ for $|z|<1$},
			xlabel={$x=\Re(z)$},
			ylabel={$y=\Im(z)$},
			colormap/viridis,
			view={0}{90},
			width=0.78\textwidth,
			height=0.68\textwidth,
			axis equal image,
			xmin=-1.15, xmax=1.15,
			ymin=-1.15, ymax=1.15,
			grid=major,
			grid style={line width=0.3pt, draw=gray!40},
			minor tick num=4,
			xtick={-1, -0.5, 0, 0.5, 1},
			ytick={-1, -0.5, 0, 0.5, 1},
			ticklabel style={font=\small},
			title style={font=\large\bfseries, align=center},
			]

			\addplot3[
			surf,
			shader=interp,
			variable=r,
			variable y=theta,
			domain=0:0.995,
			samples=45,
			domain y=0:360,
			samples y=90,
			] (
			{r*cos(theta)},
			{r*sin(theta)},
			{ min( (1-r^2)*(1+r^2-r*cos(theta))
				/ ( (1-r*cos(theta)+r^2*cos(theta)^2-r^2*sin(theta)^2)^2
				+ (2*r^2*cos(theta)*sin(theta)-r*sin(theta))^2 ), 4) }
			);

			\addplot3[domain=0:360, samples=120, samples y=1, thick, black, line width=1.5pt]
			({cos(x)},{sin(x)},{5});

			\addplot3[domain=0:360, samples=80, samples y=1, dashed, thick, gray!60]
			({0.5*cos(x)},{0.5*sin(x)},{5});

			\draw[->, thick, black!70, >=stealth] 
			(axis cs:-1.12,0,5) -- (axis cs:1.12,0,5);
			\draw[->, thick, black!70, >=stealth] 
			(axis cs:0,-1.12,5) -- (axis cs:0,1.12,5);

			\node[black!70, font=\small] at (axis cs:1.18,0,5) {$x$};
			\node[black!70, font=\small] at (axis cs:0,1.18,5) {$y$};

			\addplot3[only marks, mark=*, mark size=5pt, thick, red]
			coordinates {(0.5,0.8660254,5) (0.5,-0.8660254,5)};

			\node[red, font=\small\bfseries, anchor=south west] at (axis cs:0.52,0.92,5) {$e^{i\pi/3}$};
			\node[red, font=\small\bfseries, anchor=north west] at (axis cs:0.52,-0.92,5) {$e^{-i\pi/3}$};

			\node[align=left, font=\scriptsize, draw=red!70, fill=white, fill opacity=0.85,
			text opacity=1, rounded corners, inner sep=3pt, anchor=south]
			at (axis cs:-0.55,0.62,5)
			{$\Re\{h(z)\}\to 0$};
			\node[align=left, font=\scriptsize, draw=red!70, fill=white, fill opacity=0.85,
			text opacity=1, rounded corners, inner sep=3pt, anchor=north]
			at (axis cs:-0.55,0.62,5)
			{as $z\to e^{\pm i\pi/3}$};

			\node[black!50, font=\small] at (axis cs:0,0,5) {$0$};

			\draw[thick, black, line width=1.5pt] (axis cs:-0.95,-1.08,5) -- (axis cs:-0.75,-1.08,5);
			\node[black, font=\small] at (axis cs:-0.55,-1.08,5) {Unit circle};
			
			\draw[dashed, thick, gray!60] (axis cs:-0.95,-1.18,5) -- (axis cs:-0.75,-1.18,5);
			\node[black, font=\small] at (axis cs:-0.55,-1.18,5) {$|z|=1/2$};
			
			\draw[red, mark=*, mark size=2pt] (axis cs:-0.95,-1.28,5) -- (axis cs:-0.75,-1.28,5);
			\node[red, font=\small] at (axis cs:-0.55,-1.28,5) {Singularities of $h$};
			
		\end{axis}
	\end{tikzpicture}
	\caption[Positivity of $\Re\{h(z)\}$ and sharpness of Corollary~\ref{Cr-31}]{%
		\textbf{Real part of $h(z)=\dfrac{1-z^2}{1-z+z^2}$} on $\mathbb{U}$, arising from
		the extremal function $f_1$ via $1+\dfrac{zf_1''(z)}{f_1'(z)}=h(z)$. Positivity of
		$\Re\{h(z)\}$ confirms $f_1\in\mathcal{K}(\mathbb{U})$ (Corollary~2, \cite{Rop-Su}), hence
		$F=zf(z)\in\mathcal{K}_B(\mathbb{B})$; the decay $\Re\{h(z)\}\to 0$ at $z=e^{\pm i\pi/3}$
		(red markers) shows the bound $1/8$ in Corollary~\ref{Cr-31} is sharp.%
	}
	\label{fig:real-part-p}
\end{figure}

Therefore
\begin{align*}
	\left|\frac{l_{z}(D^2F(0)(z^2))}{2!\|z\|^2}\frac{l_{z}(D^4F(0)(z^4))}{4!\|z\|^4}-\left(\frac{l_z(D^3F(0)(z^3))}{3!\|z\|^3}\right)^2\right|=\frac{1}{8},
\end{align*}

which establishes the sharpness of the bound.

\begin{theo}\label{Th-32}Let $f\in\mathcal{H}(\mathbb{U}^n,\mathbb{C})$ with $f(0)=1$ and $F(z)=zf(z)\in\mathcal{K}_B(\mathbb{U}^n)$ and let $\mu\in [0,1]$. Then
\begin{align*}
\displaystyle
&\left\|\frac{1}{3!}D^3F(0)\left(z^2,\frac{D^3F(0)(z^3)}{3!}\right)-\mu\frac{1}{4!}D^4F(0)\left(z^3,\frac{D^2F(0)(z^2)}{2!}\right)\right\|_{\infty}\\\leq&
\max\left\{\frac{1}{9}, \frac{\mu}{8}, 1-\mu\right\}||z||_{\infty}^6, \quad z\in\mathbb{U}^n.
\end{align*}
\end{theo}

\begin{cor}\label{Cr-32}Let $f\in\mathcal{H}(\mathbb{U}^n,\mathbb{C})$ with $f(0)=1$ and $F(z)=zf(z)\in\mathcal{K}_B(\mathbb{U}^n)$. Then
\begin{align*}
\displaystyle
\left\|\frac{1}{3!}D^3F(0)\left(z^2,\frac{D^3F(0)(z^3)}{3!}\right)-\frac{1}{4!}D^4F(0)\left(z^3,\frac{D^2F(0)(z^2)}{2!}\right)\right\|_{\infty}\leq \frac{1}{8}||z||_{\infty}^6
\end{align*}
for all $z\in\mathbb{U}^n$.
\end{cor}

To establish the sharpness of the bound of Corollary \ref{Cr-32}, let us consider the function $F$ defined by $	F(z)=zf(z)$, $z\in\mathbb{U}^n$,
where 
\begin{align*}
f(z)=\frac{1}{(\omega^2-\omega)}\frac{1}{z_1}\log\frac{1+\omega^2z_1}{1+\omega z_1}=1+\frac{z_1}{2}-\frac{z_1^3}{4}-\frac{z_1^4}{5}+\frac{z_1^6}{7}+\frac{z_1^7}{8}+\cdots
\end{align*}
with $\omega=(1+i\sqrt{3})/2$. 
Clearly $F\in\mathcal{K}_B(\mathbb{U}^n)$. 
Using \eqref{Th-2x}, we get
\begin{align*}
	\frac{l_z(D^2F(0)(z^2))}{2!}=\frac{z_1||z||}{2},\;\;
	\frac{l_z(D^3F(0)(z^3))}{3!}=0\;\;\text{and}\;\;
	\frac{l_z(D^4F(0)(z^4))}{4!}=\frac{z_1^3||z||}{4}.
\end{align*}

For $0<r<1$, choosing $z=(r,0,\ldots,0)$; one can easily verify that
\begin{align*}
	\displaystyle
	\left\|\frac{1}{3!}D^3F(0)\left(z^2,\frac{D^3F(0)(z^3)}{3!}\right)-\frac{1}{4!}D^4F(0)\left(z^3,\frac{D^2F(0)(z^2)}{2!}\right)\right\|_{\infty}= \frac{1}{8}||z||_{\infty}^6,
\end{align*}
which establishes the sharpness of the bound.

\subsection{{Auxiliary lemmas}}
The lemma below is instrumental in establishing our main results.

\begin{lem}\label{Lm3}\emph{\cite[Lemma 2.4]{CKL1}} If $h\in\mathcal{P}$ is of the form $h(z)=1+\sum_{n=1}^{\infty}h_nz^n$, then
	\bea
	\label{e2.2}h_1 =2u_1,\eea
	\bea\label{e2.3} h_2=2u_1^2 + 2(1 - u_1^2)u_2,\eea
	and
	\bea
	\label{e2.4} h_3 = 2u_1^3+4(1-u_1^2)u_1u_2 - 2(1 - u_1^2)u_1u_2^2 + 2(1 - u_1^2)(1 - |u_2|^2)u_3
	\eea
	for some $u_1, u_2, u_3 \in{\ol {\mathbb{U}}}:= \{z \in \mathbb{C}: |z| \leq 1 \}$.\\
	For $u_1 \in \mathbb{T}:= \{z \in \mathbb{C}: |z| = 1 \}$, there is a unique function $h\in \mathcal{P}$ with $h_1$  as in \emph{(\ref{e2.2})}, namely
	\[h(z)=\frac{1+u_1 z}{1 - u_1 z}, \quad z \in \mathbb{U}.\]\\
	For $u_1 \in \mathbb{U}$ and $u_2 \in \mathbb{T}$, there is a unique function $h\in \mathcal{P}$ with $h_1$ and $h_2$ as in \emph{(\ref{e2.2})} and \emph{(\ref{e2.3})}, namely
	\[h(z) = \frac{1+(\ol u_1 u_2+u_1)z + u_2 z^2}{1+(\ol u_1 u_2 - u_1)z - u_2 z^2}, \quad z \in \mathbb{U}.\]\\
	For $u_1, u_2 \in \mathbb{U}$ and $u_3 \in \mathbb{T}$, there is a unique function $h\in\mathcal{P}$ with $h_1$, $h_2$ and $h_3$  as in \emph{(\ref{e2.2})--(\ref{e2.4})}, namely
	\[h(z)=\frac{1 + (\ol u_2 u_3 + \ol u_1 u_2 + u_1)z+(\ol u_1 u_3+u_1\ol u_2 u_3+u_2)z^2+u_3z^3}{1+(\ol u_2 u_3+\ol u_1 u_2-u_1)z+(\ol u_1 u_3-u_1\ol u_2 u_3-u_2)z^2-u_3z^3},\;\;z\in\mathbb{U}.\]
\end{lem}
\begin{lem}\label{l2.5}\emph{\cite{CKS1}} Let $A$, $B$, $C$  be real numbers and let
\begin{align*}
\Psi(A, B, C):= \max\limits_{z\in \ol{{\mathbb{U}}}}\left\lbrace |A+Bz+Cz^2|+1-|z|^2\right\rbrace.
\end{align*}
	
\begin{enumerate} 
\item[\emph{(i)}] If $AC\geq 0$, then
\begin{align*}
\Psi(A, B, C) =
\begin{cases}
|A|+|B|+|C|, & \text{if}\;\;\; |B|\geq 2(1-|C|), \\
1+|A|+\frac{B^2}{4(1-|C|)}, &\text{if}\;\;\; |B|<2(1-|C|).
\end{cases}
\end{align*}
\item[\emph{(ii)}] If $AC<0$, then 
\begin{align*}
\Psi(A,B,C)=
\begin{cases}
1-|A|+\frac{B^2}{4(1-|C|)}, &\text{if}\;\;\; -4AC(C^{-2}-1) \leq B^2\; \text{and}\; |B|<2(1-|C|), \\
1+|A|+\frac{B^2}{4(1+|C|)}, &\text{if}\;\;\; B^2<\min\left\{4(1+|C|)^2, -4AC(C^{-2}-1) \right\}, \\
R(A,B,C), &\text{otherwise},
\end{cases}
\end{align*}
where
\begin{align*}
R(A,B,C):=
\begin{cases}
|A|+|B|-|C|, & \text{if}\;\;\; |C|(|B|+4|A|) \leq |AB|, \\
-|A|+|B|+|C|, & \text{if}\;\;\; |AB|\leq |C|(|B|-4|A|), \\
(|C|+|A| )\sqrt{1-\frac{B^2}{4AC}}, &\text{otherwise}.
\end{cases}
\end{align*}
\end{enumerate} 
\end{lem}

\begin{proof}[\bf Proof of Theorem \ref{Th-31}]
By the given condition, we have $F\in\mathcal{K}_B(\mathbb{B})$ and so using \ref{C1}, we must have
\begin{align*}
\displaystyle
\Re\left\lbrace \frac{l_z\left([DF(z)]^{-1}\phi(z)\right)}{||z||}\right\rbrace>0, \quad z\in\mathbb{B}\backslash\{0\},\quad l_z\in T(z),
\end{align*}
where $\phi(z)=D^2 F(z)(z^2)+DF(z) z$.		
Let $z\in \mathbb{B}\backslash\{0\}$ be fixed and let $z_0=z/\|z\|$. We define the function $g:\mathbb{U}\to \mathbb{C}$ as follows
\begin{align}\label{Th-31.0}	
\displaystyle
g(\zeta)=\begin{cases}
\frac{l_{z}\left([DF(\zeta z_0)]^{-1}\phi(\zeta z_0)\right)}{\zeta},& \zeta\neq 0,\\[8pt]
1& \zeta=0.
\end{cases}
\end{align}
	
Clearly $g\in\mathcal{H}(\mathbb{U})$ and $g(0)=1$.
Therefore for $\zeta\in\mathbb{U}\backslash\{0\}$, we have
\begin{align*}
\displaystyle
g(\zeta)=\frac{l_{\zeta z_0}\left([DF(\zeta z_0)]^{-1}\phi(\zeta z_0)\right)}{||\zeta z_0||}.
\end{align*}	
Consequently $\Re\{g(\zeta)\}>0$ for all $\zeta\in\mathbb{U}$ and so we may assume that
\begin{align}\label{Th-31.1}
\displaystyle
g(\zeta)=1+h_1\zeta+h_2\zeta^2+\ldots
\end{align}
	
Since $F(z)$ is normalized, it has the power series expansion
\begin{align}\label{Th-31.2}
\displaystyle
F(z)=z+\frac{1}{2!}D^2 F(0)(z^2)+\frac{1}{3!}D^3 F(0)(z^3)+\frac{1}{4!}D^4 F(0)(z^4)+\ldots.
\end{align}
	
Clearly
\begin{align}\label{Th-31.3}
\displaystyle
D F(z)=I+D^2 F(0)(z)+\frac{1}{2}D^3 F(0)(z^2)+\frac{1}{6}D^4F(0)(z^3)+\ldots
\end{align}
and 
\begin{align}\label{Th-31.4}
\displaystyle
D^2 F(z)=D^2 F(0)+D^3 F(0)(z)+\frac{1}{2}D^4F(0)(z^2)+\ldots
\end{align}
	
Using \eqref{Th-31.3} and \eqref{Th-31.4}, we get
\begin{align}\label{Th-31.5}
\phi(z)=z+2D^2F(0)(z^2)+\frac{3}{2}D^3F(0)(z^3)+\frac{2}{3}D^4F(0)(z^4)+\ldots	
\end{align}
		
Now substituting \eqref{Th-2x} into \eqref{Th-31.5}, we get
\begin{align}\label{Th-31.6}
\phi(z)=z+4Df(0)(z)z+\frac{9}{2}D^2f(0)(z^2)z+\frac{8}{3}D^3f(0)(z^3)z+\ldots	
\end{align}
	
On the other hand, \eqref{Th-21.2b} gives
\begin{align}\label{Th-31.7}
\displaystyle
[D F(z)]^{-1}\phi(z)=\frac{\phi(z)}{f(z)+D f(z)z}.
\end{align}
	
Therefore using \eqref{Th-31.6} to \eqref{Th-31.7}, we obtain
\begin{align}\label{Th-31.8}
\displaystyle
&[D F(\zeta z_0)]^{-1}\phi(\zeta z_0)\\=&
\frac{\zeta z_0+4Df(0)(\zeta z_0)\zeta z_0+\frac{9}{2}D^2f(0)((\zeta z_0)^2)\zeta z_0+\frac{8}{3}D^3f(0)((\zeta z_0)^3)\zeta z_0+\ldots}{f(\zeta z_0)+D f(\zeta z_0)\zeta z_0}.\nonumber	
\end{align}
	
Also, from \eqref{Th-21.5} and \eqref{Th-21.6}, we obtain
\begin{align}\label{Th-31.9}
f(\zeta z_0)+D f(\zeta z_0)\zeta z_0=&1+2Df(0)(z_0)\zeta+\frac{3}{2}D^2f(0)(z_0^2)\zeta^2\\&+\frac{2}{3}D^3f(0)(z_0^3)\zeta^3+\ldots\nonumber
\end{align}

In view of \eqref{Th-31.0} and using \eqref{Th-31.8} and \eqref{Th-31.9}, we find that
\begin{align}\label{Th-31.10}
\displaystyle
&\left(1+h_1\zeta+h_2\zeta^2+h_3\zeta^3+\ldots\right)\times\\&\left(1+2Df(0)(z_0)\zeta+\frac{3}{2}D^2f(0)(z_0^2)\zeta^2+\frac{2}{3}D^3f(0)(z_0^3)\zeta^3+\ldots\right)\nonumber\\=&
1+4Df(0)(z_0)\zeta+\frac{9}{2} D^2f(0)(z_0^2)\zeta^2+\frac{8}{3}D^3f(0)(z_0^3)\zeta^3+\ldots.\nonumber
\end{align}
	
Clearly from \eqref{Th-31.10}, we find that
\begin{align}\label{Th-31.11}
\displaystyle
\begin{cases}
h_1&=2Df(0)(z_0),\\[8pt]
h_2&=3D^2f(0)(z_0^2)-4\left(Df(0)(z_0)\right)^2,\\[8pt]
h_3&=2D^3f(0)(z_0^3)-9Df(0)(z_0)D^2f(0)(z_0^2)+8(Df(0)(z_0))^3.
\end{cases}
\end{align}
	
Using \eqref{Th-2y}, we deduce from \eqref{Th-31.11} that
\begin{align}\label{Th-31.12}
\displaystyle
\frac{l_z(D^2F(0)(z^2))}{2!\|z\|^2}=\frac{1}{2}h_1,
\end{align}
\begin{align}\label{Th-31.13}
\displaystyle
\frac{l_{z}(D^3F(0)(z^3))}{3!\|z\|^3}=\frac{1}{6}\left(h_2+h_1^2\right)
\end{align}
and
\begin{align}\label{Th-31.14}
\displaystyle
\frac{l_{z}(D^4F(0)(z^4))}{4!\|z\|^4}=\frac{1}{12}\left(h_3+\frac{3}{2}h_1h_2+\frac{1}{2}h_1^3\right).
\end{align}
	
Now using \eqref{Th-31.12}-\eqref{Th-31.14}, we find that 
\begin{align}\label{e3.17}
\nonumber&\left|\mu\frac{l_{z}(D^2F(0)(z^2))}{2!\|z\|^2}\frac{l_{z}(D^4F(0)(z^4))}{4!\|z\|^4}-\left(\frac{l_z(D^3F(0)(z^3))}{3!\|z\|^3}\right)^2\right|\\=&\frac{1}{144}\left|6\mu h_1h_3+(9\mu-8)h_1^2h_2+(3\mu-4)h_1^4-4h_2^2 \right|.
\end{align}
Since $\mu\in[0, 1]$, we consider the following cases:\\
{\bf Case-1.} Choose $\mu=0$, so from \eqref{e3.17} we have the trivial bound 
\begin{align*}
	\left|\mu\frac{l_{z}(D^2F(0)(z^2))}{2!\|z\|^2}\frac{l_{z}(D^4F(0)(z^4))}{4!\|z\|^4}-\left(\frac{l_z(D^3F(0)(z^3))}{3!\|z\|^3}\right)^2\right|=1.
\end{align*}
{\bf Case-2.} Choose $0<\mu\le1$. Now, substituting \eqref{e2.2}-\eqref{e2.4} into \eqref{e3.17} we obtain
\begin{align}\label{e3.18}
	\nonumber&\left|\mu\frac{l_{z}(D^2F(0)(z^2))}{2!\|z\|^2}\frac{l_{z}(D^4F(0)(z^4))}{4!\|z\|^4}-\left(\frac{l_z(D^3F(0)(z^3))}{3!\|z\|^3}\right)^2\right|\\=&\frac{\mu}{6}\bigg|-\frac{6(1-\mu)}{\mu}u_1^4+\left(\frac{5\mu-4}{\mu}\right)(1-u_1^2)u_1^2u_2-\left(\frac{2+(3\mu-2)u_1^2}{3\mu}\right)(1-u_1^2)u_2^2\nonumber\\&+(1-u_1^2)(1-|u_2|^2)u_1u_3 \bigg|
\end{align}

Since $u_1\in[0,1]$, by Lemma~\ref{Lm3}, we first consider the boundary value of $u_1$.
\begin{align*} &\left|\mu\frac{l_{z}(D^2F(0)(z^2))}{2!\|z\|^2}\frac{l_{z}(D^4F(0)(z^4))}{4!\|z\|^4}-\left(\frac{l_z(D^3F(0)(z^3))}{3!\|z\|^3}\right)^2\right|\\&=
\begin{cases}
\frac{1}{9}\mid u_2\mid^2\le \frac{1}{9}, & u_1=0, \\[2mm]
1-\mu, & u_1=1.
\end{cases} \end{align*} 

It therefore remains to consider the case $0<u_1<1$.
Also, by applying the triangle inequality to (\ref{e3.18}), for $u_1\in (0, 1)$ and $\mid u_3\mid\le1$, we obtain the following inequality
\begin{align}\label{e3.19}
&\left|\mu\frac{l_{z}(D^2F(0)(z^2))}{2!\|z\|^2}\frac{l_{z}(D^4F(0)(z^4))}{4!\|z\|^4}-\left(\frac{l_z(D^3F(0)(z^3))}{3!\|z\|^3}\right)^2\right|\nonumber\\\le&\frac{\mu}{6}u_1(1-u_1^2)\bigg(\bigg|-\frac{6(1-\mu)}{\mu(1-u_1^2)}u_1^3+\left(\frac{5\mu-4}{\mu}\right)u_1u_2-\left(\frac{2+(3\mu-2)u_1^2}{3\mu u_1}\right)u_2^2\bigg|\nonumber\\&+(1-|u_2|^2)\bigg)\nonumber\\&\le\frac{\mu}{6}u_1(1-u_1^2) \Psi(A, B, C),
\end{align}
where $\Psi(A, B, C)=\mid A+Bu_2+Cu_2^2\mid+(1-\mid u_2\mid^2)$ with
 \begin{align}\label{e3.20} A=-\frac{6(1-\mu)}{\mu(1-u_1^2)}u_1^3,\;\;  B=\left(\frac{5\mu-4}{\mu}\right)u_1\;\; \text{and}\;\; C=-\left(\frac{2+(3\mu-2)u_1^2}{3\mu u_1}\right).
 \end{align}
 

Since $0<\mu\le 1$, it is easy to verify that $AC>0$. Also, from \eqref{e3.20} we obtain 
\begin{align*} |B|-2(1-|C|)=
\begin{cases}
\frac{(8-9\mu)u_1^2-6\mu u_1+4}{3\mu u_1}, & 0<\mu<\frac{4}{5}, \\[2mm]
\frac{(21\mu-16)u_1^2-6\mu u_1+4}{3\mu u_1}, & \frac{4}{5}\le\mu\le1.
\end{cases} \end{align*}

In view of $0<\mu\le 1$ and $0<u_1<1$, it is easy to verify that $|B|\ge 2(1-|C|)$. Therefore, by Lemma \ref{l2.5}, we get
\begin{align}\label{e3.21}
\Psi(A, B, C)&=|A|+|B|+|C|=
\begin{cases}
\frac{(8-6\mu)u_1^4+(8-12\mu)u_1^2+2}{3\mu u_1(1-u_1^2)}, & 0<\mu<\frac{4}{5}, \\[2mm]
\frac{(32-36\mu)u_1^4+(18\mu-16)u_1^2+2}{3\mu u_1(1-u_1^2)}, & \frac{4}{5}\le\mu\le1.
\end{cases} 
\end{align}

Now, we consider the following cases:
\medskip

{\bf A.} Let us consider $0<\mu<\frac{4}{5}$. Then from \eqref{e3.19} and \eqref{e3.21}, we get 
\begin{align}\label{e3.22}
\nonumber&\left|\mu\frac{l_{z}(D^2F(0)(z^2))}{2!\|z\|^2}\frac{l_{z}(D^4F(0)(z^4))}{4!\|z\|^4}-\left(\frac{l_z(D^3F(0)(z^3))}{3!\|z\|^3}\right)^2\right|\\\le&\frac{1}{9}\left((4-3\mu)u_1^4+(4-6\mu)u_1^2+1\right)=\frac{1}{9}P(u_1),
\end{align}
where $P(u_1)=(4-3\mu)u_1^4+(4-6\mu)u_1^2+1$.
\begin{figure}[!ht]
	\centering
	\begin{tikzpicture}
		\begin{axis}[
			title={Case A: Behavior of $P(x)=(4-3\mu)x^2+(4-6\mu)x+1$},
			xlabel={$x$},
			ylabel={$P(x)$},
			xmin=0,
			xmax=1.05,
			ymin=0,
			ymax=7,
			xtick={0,0.25,0.5,0.75,1},
			ytick={0,1,2,3,4,5,6,7},
			grid=both,
			grid style={line width=.1pt, draw=gray!20},
			major grid style={line width=.2pt, draw=gray!50},
				legend style={
					at={(0.03,0.72)},
					anchor=north west,
					font=\small
				},
			axis line style={thick},
			width=0.8\textwidth,
			height=0.5\textwidth,
			samples=100,
			]
			
			\addplot[
			domain=0:1,
			blue,
			line width=2pt
			]
			{(4-3*0.3)*x^2+(4-6*0.3)*x+1};
			\addlegendentry{$\mu=0.3$}
			
			\addplot[
			domain=0:1,
			red,
			line width=2pt,
			dashed
			]
			{(4-3*0.7)*x^2+(4-6*0.7)*x+1};
			\addlegendentry{$\mu=0.7$}
			
			\addplot[
			only marks,
			mark=*,
			blue,
			mark size=2.5pt
			]
			coordinates {(1,6.3)};
			
			\addplot[
			only marks,
			mark=*,
			red,
			mark size=2.5pt
			]
			coordinates {(1,2.7)};
			
			\addplot[
			domain=0:1,
			blue,
			line width=1.2pt,
			dashed
			]
			{6.3};
			
			\addplot[
			domain=0:1,
			red,
			line width=1.2pt,
			dotted
			]
			{2.7};
			
			\node[
			anchor=east,
			blue,
			font=\small
			] at (axis cs:0.98,6.55)
			{$P(1)=9(1-\mu)=6.3$};
			
			\node[
			anchor=east,
			red,
			font=\small
			] at (axis cs:0.98,2.95)
			{$P(1)=9(1-\mu)=2.7$};
			
		\end{axis}
	\end{tikzpicture}
	
	\caption{Illustration of Case A ($0<\mu<4/5$). For the representative values
		$\mu=0.3$ and $\mu=0.7$, the function $P(x)$ is convex on $[0,1]$ and attains
		its maximum at $x=1$, where $P(1)=9(1-\mu)$.}
	\label{fig:case-a}
\end{figure}
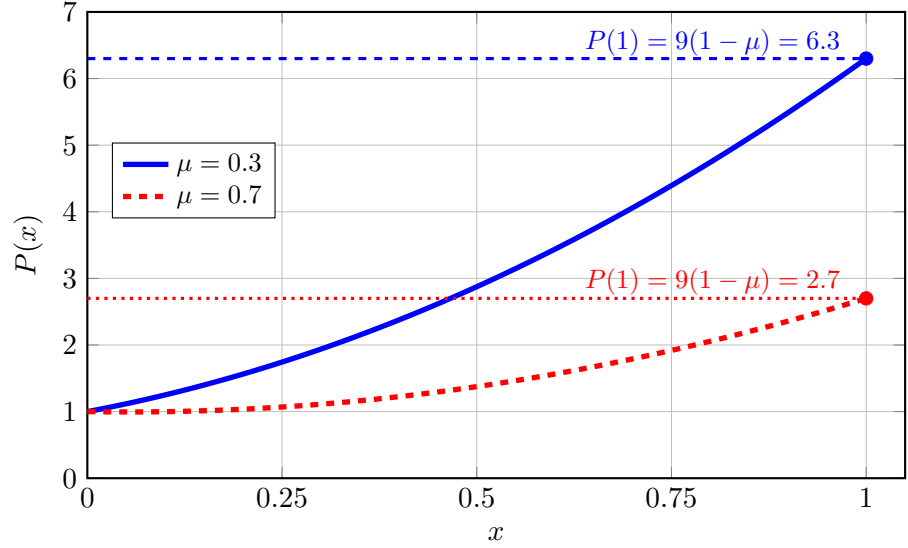
 Let us choose $u_1^2=x$ and we have $P(x)=(4-3\mu)x^2+(4-6\mu)x+1$. Since, $0<u_1<1$, then it is clear that $0<x<1$. Here it is easy to verify that ${P(x)}\le P(1)=9(1-\mu)$. Therefore, from \eqref{e3.22} we obtain 
\begin{align*}
\left|\mu\frac{l_{z}(D^2F(0)(z^2))}{2!\|z\|^2}\frac{l_{z}(D^4F(0)(z^4))}{4!\|z\|^4}-\left(\frac{l_z(D^3F(0)(z^3))}{3!\|z\|^3}\right)^2\right|\le 1-\mu.
\end{align*}

\medskip
{\bf B.} Let us consider $\frac{4}{5}\le\mu\le 1$. Here from \eqref{e3.19} and \eqref{e3.21}, we obtain 
\begin{align}\label{e3.23}
\nonumber&\left|\mu\frac{l_{z}(D^2F(0)(z^2))}{2!\|z\|^2}\frac{l_{z}(D^4F(0)(z^4))}{4!\|z\|^4}-\left(\frac{l_z(D^3F(0)(z^3))}{3!\|z\|^3}\right)^2\right|\\\le&\frac{1}{9}\left((16-18\mu)u_1^4+(9\mu-8)u_1^2+1\right)=\frac{1}{9}Q(u_1),
\end{align}
where 
\begin{align*}
Q(u_1)=(16-18\mu)u_1^4+(9\mu-8)u_1^2+1.
\end{align*}
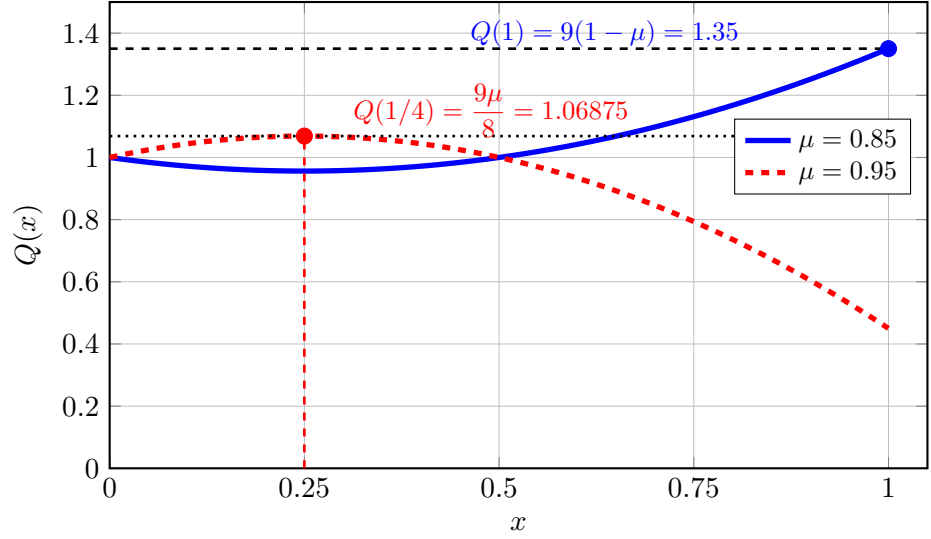
\begin{figure}[!ht]
	\centering
	\begin{tikzpicture}
		\begin{axis}[
			title={Case B: Behavior of $Q(x)=(16-18\mu)x^2+(9\mu-8)x+1$},
			xlabel={$x$}, ylabel={$Q(x)$},
			xmin=0, xmax=1.05, ymin=0, ymax=1.5,
			xtick={0,0.25,0.5,0.75,1},
			grid=both,
			grid style={line width=.1pt, draw=gray!20},
			major grid style={line width=.2pt, draw=gray!50},
			legend style={at={(0.98,0.75)}, anchor=north east, font=\small},
			axis line style={thick},
			width=0.8\textwidth, height=0.5\textwidth, samples=100,
			]

			\addplot[domain=0:1, blue, line width=2pt]
			{(16-18*0.85)*x^2+(9*0.85-8)*x+1};
			\addlegendentry{$\mu=0.85$}

			\addplot[domain=0:1, red, dashed, line width=2pt]
			{(16-18*0.95)*x^2+(9*0.95-8)*x+1};
			\addlegendentry{$\mu=0.95$}

			\addplot[only marks, mark=*, blue, mark size=3pt]
			coordinates {(1,1.35)};
			\addplot[domain=0:1, black, dashed, line width=1pt] {1.35};
			\node[anchor=west, blue, font=\small] at (axis cs:0.45,1.40)
			{$Q(1)=9(1-\mu)=1.35$};

			\addplot[only marks, mark=*, red, mark size=3pt]
			coordinates {(0.25,1.06875)};
			\draw[red, dashed, line width=1pt]
			(axis cs:0.25,0) -- (axis cs:0.25,1.06875);
			\addplot[domain=0:1, black, dotted, line width=1pt] {1.06875};
			\node[anchor=west, red, font=\small] at (axis cs:0.30,1.15)
			{$Q(1/4)=\dfrac{9\mu}{8}=1.06875$};
			
		\end{axis}
	\end{tikzpicture}
	\caption{Illustration of Case B ($4/5\le\mu\le1$). For $\mu=0.85\in[4/5,8/9]$, $Q$ is convex and its maximum on $[0,1]$ occurs at the boundary $x=1$, giving $Q(1)=9(1-\mu)$. For $\mu=0.95\in[8/9,1]$, $Q$ is concave and its maximum occurs at the interior critical point $x=1/4$, giving $Q(1/4)=9\mu/8$.}
	\label{fig:case-b-corrected}
\end{figure}
 Suppose $u_1^2=x$, then $Q(x)=(16-18\mu)x^2+(9\mu-8)x+1$. Since $0<u_1<1$ implies $0<x<1$. Thus the critical point of $Q(x)$ is $x_c=\frac{1}{4}$. Also we have 
\begin{align*}
Q^{\prime\prime}(x)=4(8-9\mu).
\end{align*}

So it is clear that  $Q^{\prime\prime}<0$ for $\frac{8}{9}<\mu\le 1$. Therefore the maximum value of $Q(x)$ is $\frac{9\mu}{8}$.
Also for $\frac{4}{5}\le \mu\le\frac{8}{9}$, maximum value of $Q(x)$ exists at the boundary. Hence
\begin{align*}
Q(x)\le Q(1)=9(1-\mu).
\end{align*}

Thus from \eqref{e3.23}, we obtain 
\begin{align*}
&\left|\mu\frac{l_{z}(D^2F(0)(z^2))}{2!\|z\|^2}\frac{l_{z}(D^4F(0)(z^4))}{4!\|z\|^4}-\left(\frac{l_z(D^3F(0)(z^3))}{3!\|z\|^3}\right)^2\right|\\\le&
\begin{cases}
1-\mu, & \frac{4}{5}\le\mu\le\frac{8}{9}, \\[2mm]	\frac{\mu}{8}, & \frac{8}{9}\le\mu\le1.
	\end{cases}
\end{align*}

Therefore, considering all the scenarios, we obtain the desired upper bound:
\begin{align*}
	\left|\mu\frac{l_{z}(D^2F(0)(z^2))}{2!\|z\|^2}\frac{l_{z}(D^4F(0)(z^4))}{4!\|z\|^4}-\left(\frac{l_z(D^3F(0)(z^3))}{3!\|z\|^3}\right)^2\right|\le\max\left\{\frac{1}{9}, \frac{\mu}{8}, 1-\mu\right\}.
\end{align*}
\end{proof}

\begin{proof}[\bf Proof of Theorem \ref{Th-32}]
By the given condition, $F\in\mathcal{K}_B(\mathbb{U}^n)$. Therefore
\begin{align*}
	\displaystyle
	\Re\left\lbrace \frac{l_z\left([DF(z)]^{-1}\phi(z)\right)}{||z||_{\infty}}\right\rbrace>0, \quad z\in\mathbb{U}^n\backslash\{0\},\quad l_z\in T(z),
\end{align*}
where $\phi(z)=D^2 F(z)(z^2)+DF(z) z$.

Let $w(z)=(w_1(z),w_2(z),\ldots,w_n(z))'=[D F(z)]^{-1}F(z)$ be a column vector in $\mathbb{C}^n$.  Clearly
\begin{align}\label{Th-32.3}
\displaystyle
\Re\left\lbrace \frac{w_k(z)}{z_k}\right\rbrace>0,\quad z\in\mathbb{U}^n\backslash\{0\},
\end{align}
where $w_k(z)$ is the $k$-th coordinate of $w(z)$ such that $||z||_{\infty}=|z_k|$.
	
Let $z\in\mathbb{U}^n\backslash\{0\}$ and let $z_0=z/\|z\|$. We define the function $g_k:\mathbb{U}\to \mathbb{C}$ as follows
\begin{align}\label{Th-32.4}	
\displaystyle
g_k(\zeta)=\begin{cases}
\frac{w_k(\zeta z_0)}{\zeta z_k},& \zeta\neq 0,\\[8pt]
1& \zeta=0.
\end{cases}
\end{align}
	
Now it follows from \eqref{Th-32.3} and \eqref{Th-32.4} that $\Re\{g_k(\zeta)\}>0$ for all $\zeta\in\mathbb{U}$. 	
On the other hand, \eqref{Th-31.8} gives
\begin{align}\label{Th-32.6}
\displaystyle
&g_k(\zeta)\\=&\frac{\zeta z_0+4Df(0)(\zeta z_0)\zeta z_0+\frac{9}{2}D^2f(0)((\zeta z_0)^2)\zeta z_0+\frac{8}{3}D^3f(0)((\zeta z_0)^3)\zeta z_0+\ldots}{f(\zeta z_0)+D f(\zeta z_0)\zeta z_0}.\nonumber
\end{align}
		
Let
\begin{align}\label{Th-32.7a}
g_k(\zeta)=1+g_{k1}\zeta+g_{k2}\zeta^2+g_{k3}\zeta^3+\ldots
\end{align}
	
Now using \eqref{Th-31.9} and \eqref{Th-32.7a} to \eqref{Th-32.6} and then comparing coefficients, we obtain
\begin{align}\label{Th-32.8}
\begin{cases}
Df(0)(z_0)&=\frac{1}{2}g_{k1},\\[8pt]
D^2f(0)(z_0^2)&=\frac{1}{3}\left(h_{k2}+h_{k1}^2\right),\\[8pt]
D^3f(0)(z_0^3)&=\frac{1}{2}\left(g_{k3}+\frac{3}{2}g_{k1}g_{k2}+\frac{1}{2}g_{k1}^2\right).
\end{cases}
\end{align}
	
Let $F(z)=(F_1(z),F_2(z),\ldots,F_n(z))$. 
In view of \eqref{Th-22.11}, we obtain
\begin{align}\label{Th-32.12}
\displaystyle
&\left|\frac{1}{3!}D^3F_k(0)\left(z_0^2,\frac{D^3F(0)(z_0^3)}{3!}\right)\frac{||z||_{\infty}}{z_k}-\mu\frac{1}{4!}D^4F_k(0)\left(z_0^3,\frac{D^2F(0)(z_0^2)}{2!}\right)\frac{||z||_{\infty}}{z_k}\right|\nonumber\\=&
\left|\left(\frac{D^2f(0)(z_0^2)}{2!}\right)^2-\mu Df(0)(z_0)\frac{D^3f(0)(z_0^3)}{3!}\right|.
\end{align}
	
Using \eqref{Th-32.8} to \eqref{Th-32.12}, we obtain
\begin{align*}
\displaystyle
&\left|\frac{1}{3!}D^3F_k(0)\left(z_0^2,\frac{D^3F(0)(z_0^3)}{3!}\right)\frac{||z||_{\infty}}{z_k}-\mu\frac{1}{4!}D^4F_k(0)\left(z_0^3,\frac{D^2F(0)(z_0^2)}{2!}\right)\frac{||z||_{\infty}}{z_k}\right|\\=&
\frac{1}{144}\left|6\mu g_{k1}g_{k3}+(9\mu-8)g_{k1}^2g_{k2}+(3\mu-4)g_{k1}^4-4g_{k2}^2 \right|.	
\nonumber
\end{align*}
	
Using similar arguments as in the proof of Theorem \ref{Th-31}, we have
\begin{align*}
\displaystyle
&\left|\frac{1}{3!}D^3F_k(0)\left(z_0^2,\frac{D^3F(0)(z_0^3)}{3!}\right)\frac{||z||_{\infty}}{z_k}-\mu\frac{1}{4!}D^4F_k(0)\left(z_0^3,\frac{D^2F(0)(z_0^2)}{2!}\right)\frac{||z||_{\infty}}{z_k}\right|\\&\le\max\left\{\frac{1}{9}, \frac{\mu}{8}, 1-\mu\right\}.	
\nonumber
\end{align*}

Now adopting the same procedure as done in the proof of Theorem \ref{Th-22}, we can prove that 	
\begin{align*}
\displaystyle
&\left\|\frac{1}{3!}D^3F(0)\left(z^2,\frac{D^3F(0)(z^3)}{3!}\right)-\mu\frac{1}{4!}D^4F(0)\left(z^3,\frac{D^2F(0)(z^2)}{2!}\right)\right\|_{\infty}\\\leq&
\max\left\{\frac{1}{9}, \frac{\mu}{8}, 1-\mu\right\}||z||_{\infty}^6, \quad z\in\mathbb{U}^n.
\end{align*}
\end{proof}
\noindent\textbf{Conflict of interest:} The authors declare that there is no conflict of interest regarding the publication of this paper.\vspace{1.2mm}

\noindent {\bf Funding:} Not Applicable.\vspace{1.2mm}

\noindent\textbf{Data availability statement:}  Data sharing not applicable to this article as no datasets were generated or analysed during the current study.\vspace{1.2mm}

\noindent {\bf Authors' contributions:} All the authors have equal contributions in preparation of the manuscript.

\end{document}